\documentclass[11pt]{article}
\usepackage{amsmath,amssymb,amsthm,amscd}

\newtheorem{theorem}{Theorem}[section]
\newtheorem{lemma}[theorem]{Lemma}
\newtheorem{proposition}[theorem]{Proposition}
\newtheorem{corollary}[theorem]{Corollary}

\theoremstyle{plain}

\theoremstyle{definition}
\newtheorem{definition}[theorem]{Definition}

\numberwithin{equation}{section}

\renewcommand{\labelenumi}{\textup{(\theenumi)}}

\title{On a family of $C^*$-subalgebras of Cuntz--Krieger algebras \\
}
\author{Kengo Matsumoto \\
Department of Mathematics \\
Joetsu University of Education \\
Joetsu, 943-8512, Japan
}

\begin{document}


\maketitle

\date{}

\def\det{{{\operatorname{det}}}}

\begin{abstract}
In this paper, we study a family of $C^*$-subalgebras defined by fixed points of 
generalized gauge actions of a Cuntz--Krieger algebra, 
by introducing a family of \'etale groupoids whose associated $C^*$-algebras 
are these $C^*$-subalgebras. 
We know that topological conjugacy classes of one-sided topological Markov
shifts are characterized in terms of the isomorphism classes of these \'etale groupoids.
\end{abstract}




\newcommand{\Ker}{\operatorname{Ker}}
\newcommand{\sgn}{\operatorname{sgn}}
\newcommand{\Ad}{\operatorname{Ad}}
\newcommand{\ad}{\operatorname{ad}}
\newcommand{\orb}{\operatorname{orb}}

\def\Re{{\operatorname{Re}}}
\def\det{{{\operatorname{det}}}}
\newcommand{\K}{\mathcal{K}}

\newcommand{\N}{\mathbb{N}}
\newcommand{\C}{\mathcal{C}}
\newcommand{\R}{\mathbb{R}}
\newcommand{\Rp}{{\mathbb{R}}^*_+}
\newcommand{\T}{\mathbb{T}}
\newcommand{\Z}{\mathbb{Z}}
\newcommand{\Zp}{{\mathbb{Z}}_+}
\def\AF{{{\operatorname{AF}}}}

\def\OA{{{\mathcal{O}}_A}}
\def\OAf{{{\mathcal{O}}_{A_f}}}
\def\OB{{{\mathcal{O}}_B}}
\def\SOA{{{\mathcal{O}}_A}\otimes{\mathcal{K}}}
\def\SOB{{{\mathcal{O}}_B}\otimes{\mathcal{K}}}
\def\F{{\mathcal{F}}}
\def\G{{\mathcal{G}}}
\def\FA{{{\mathcal{F}}_A}}
\def\PA{{{\mathcal{P}}_A}}
\def\FAH{{{\mathcal{F}}_{A,H}}}
\def\FAf{{{\mathcal{F}}_{A,f}}}
\def\FAg{{{\mathcal{F}}_{A,g}}}
\def\FBg{{{\mathcal{F}}_{B,g}}}
\def\FAb{{{\mathcal{F}}_A^b}}
\def\FB{{{\mathcal{F}}_B}}
\def\DA{{{\mathcal{D}}_A}}
\def\DB{{{\mathcal{D}}_B}}
\def\Ext{{{\operatorname{Ext}}}}
\def\Max{{{\operatorname{Max}}}}
\def\Per{{{\operatorname{Per}}}}
\def\Homeo{{{\operatorname{Homeo}}}}
\def\Out{{{\operatorname{Out}}}}
\def\Aut{{{\operatorname{Aut}}}}
\def\Int{{{\operatorname{Int}}}}
\def\Ad{{{\operatorname{Ad}}}}
\def\Inn{{{\operatorname{Inn}}}}
\def\det{{{\operatorname{det}}}}
\def\exp{{{\operatorname{exp}}}}
\def\nep{{{\operatorname{nep}}}}
\def\sgn{{{\operatorname{sign}}}}
\def\cobdy{{{\operatorname{cobdy}}}}
\def\Ker{{{\operatorname{Ker}}}}
\def\ind{{{\operatorname{ind}}}}
\def\id{{{\operatorname{id}}}}
\def\supp{{{\operatorname{supp}}}}
\def\co{{{\operatorname{co}}}}
\def\scoe{{{\operatorname{scoe}}}}
\def\coe{{{\operatorname{coe}}}}
\def\Span{{{\operatorname{Span}}}}

\def\S{\mathcal{S}}

\def\tS{\tilde{S}}

\def\coe{{{\operatorname{coe}}}}
\def\scoe{{{\operatorname{scoe}}}}
\def\uoe{{{\operatorname{uoe}}}}
\def\ucoe{{{\operatorname{ucoe}}}}
\def\event{{{\operatorname{event}}}}

\section{Introduction and Preliminary}

In this paper, we will study a family of $C^*$-subalgebras defined by fixed points of 
generalized gauge actions of a Cuntz--Krieger algebra
$\OA$
from a groupoid view point.
Each of the $C^*$-subalgebras contains  the canonical maximal abelian $C^*$-subalgebra
of $\OA$. 
They are generalization of the canonical AF subalgebra of a Cuntz--Krieger algebra.
Let $A =[A(i,j)]_{i,j=1}^N$
be an irreducible non permutation matrix with entries in $\{0,1\}$.
The Cuntz--Krieger algebra $\OA$ is defined by $N$ partial isometries 
$S_1,\dots, S_N$ satisfying the operator relations:
$1 = \sum_{j=1}^N S_j S_j^*, \, S_i^* S_i = \sum_{j=1}^N A(i,j) S_j S_j^*, i=1,\dots, N$
(\cite{CK}).
The algebras are closely related to a class of symbolic dynamical systems
called topological Markov shifts.
The one-sided topological Markov shift $(X_A,\sigma_A)$ for the matrix $A$
is defined by its shift space $X_A$ consisting of right one-sided sequences
$(x_n)_{n \in \N} \in \{1,2,\dots, N\}^\N$ satisfying 
$A(x_n, x_{n+1}) =1, n \in \N$ with its shift transformation
$\sigma_A:X_A\longrightarrow X_A$ defined by
$\sigma_A((x_n)_{n \in \N}) = (x_{n+1})_{n \in \N},$
where $\N$ denotes the set of positive integers.
The topology of $X_A$ is endowed with the relative topology 
of the infinite product topology on
$\{1,2,\dots, N\}^\N$ for the discrete set $\{1,2,\dots, N\}$.
Hence the shift space $X_A$ 
is a compact Hausdorff space homeomorphic to a Cantor discontinuum and 
the shift transformation 
$\sigma_A: X_A\longrightarrow X_A$ is a continuous surjection.
Let us denote by $B_k(X_A)$ the set of admissible words
$\{ (x_1,\dots, x_k) \in \{1,\dots, N\}^k\mid (x_n)_{n \in \N} \in X_A\}$
 of $X_A$ with its length $k$.    
Let $U_\mu$ be the cylinder set 
$\{ (x_n)_{n \in \N} \in X_A \mid x_1 = \mu_1,\dots, \mu_m = x_m \}$
for the word $\mu = (\mu_1, \dots, \mu_m)\in B_m(X_A)$.
Let $\chi_{U_\mu}$ be the characteristic function of $U_\mu$ on $X_A$.
It is well-known that 
the corresondence
$\chi_{U_\mu} \longrightarrow S_{\mu_1}\cdots S_{\mu_m}S_{\mu_m}^*\cdots S_{\mu_1}^*$
gives rise to an isomorphism from the commutative $C^*$-algebra 
$C(X_A)$ of complex valued continuous functions on $X_A$ 
to the $C^*$-subalgebra $\DA$ of $\OA$ 
generated by the projections of the form
$S_{\mu_1}\cdots S_{\mu_m}S_{\mu_m}^*\cdots S_{\mu_1}^*.$
The gauge action written $\rho^A$ of $\T$ to the automorphism group of
$\OA$ is defined by the one-parameter family of automorphisms
$\rho^A_t, t \in \T$ defined by 
the correspondence
$S_i \longrightarrow \exp(2\pi\sqrt{-1}t) S_i, t \in \R/\Z = \T, i=1,\dots,N$. 
The fixed point algebra of $\OA$ under $\rho^A$
is an AF algebra defined by the matrix 
$A$ (\cite{CK}).
It is called the standard AF-subalgebra of $\OA$ denoted by $\F_A$.
Let us denote by $C(X_A,\Z)$
the set of integer valued continuous functions 
on $X_A$.
A continuous function $f \in C(X_A,\Z)$
gives rise to an element of $C(X_A)$ and hence of $\DA$.
Since $\exp{(2\pi\sqrt{-1}t f)}$ is a unitary in $\DA$
for each $t \in \T$, 
the correspondence
$S_i \longrightarrow \exp{(2\pi\sqrt{-1}t f)}S_i, i=1,2,\dots,N, t \in \T$
yields an automorphism of $\OA$ written $\rho^{A,f}_t$.
The family of automorphisms $\rho^{A,f}_t, t \in \T$
defines an action of $\T$ on the $C^*$-algebra $\OA$.
It is called the gauge action with potential $f$, or a generalized gauge action.
For $f\equiv 1$, the action $\rho^{A,1}$ coincides with the gauge action
$\rho^A$ that is in particular called the standard gauge action on $\OA$.

In the first half of the paper, 
 we will study a family $\F_{A,f}, f \in C(X_A,\Z)$ 
 of $C^*$-subalgebras of $\OA$
introduced in \cite{MaPre2020d}.
They are defined in the following way.   
\begin{definition}[{\cite[Definition 2.5]{MaPre2020d}}]
For $f \in C(X_A, \Z)$, define a $C^*$-subalgebra $\FAf$ of $\OA$
by  the fixed point subalgebra of $\OA$ under the action $\rho^{A,f}$
\begin{equation}
\FAf : = 
\{ X \in \OA \mid \rho^{A,f}_t(X) = X \text{ for all } t \in \T\}.
\end{equation}  
We call the $C^*$-algebra $\FAf$ the {\it cocycle algebra for}\/ $f$.
\end{definition}
For constant functions $f \equiv 0$ and $f\equiv 1$,
we know that $\F_{A,0} = \OA$ and $\F_{A,1} = \F_A$, respectively.
Hence the family $\F_{A,f}, f \in C(X_A,\Z)$ 
of $C^*$-subalgebras of $\OA$ generalize both $\OA$ and $\F_A$.
In \cite{MaPre2020d}, the following result was proved. 
\begin{theorem}[{\cite[Therem 1.4]{MaPre2020d}, cf. \cite{MaPre2020b}}] \label{thm:1.2}
Let $A$ and $B$ be irreducible non permutation matrices with entries in $\{0,1\}$.
Then the one-sided topological Markov shifts $(X_A,\sigma_A)$ and $(X_B,\sigma_B)$
are topologically conjugate if and only if 
there exists an isomorphism $\Phi: \OA\longrightarrow \OB$ 
of $C^*$-algebras such that 
$\Phi(\DA) = \DB$ and
\begin{equation}
\Phi(\F_{A,f}) = \F_{B,\Phi(f)} \quad \text{ for all } f \in C(X_A,\Z),
\end{equation}
where $\Phi(f) \in C(X_B, \Z)$ for $f \in C(X_A,\Z)$.
\end{theorem}
Since $\F_{A,0} = \OA, \F_{A,1} = \FA$ and 
$\cap_{f \in C(X_A,\Z)} \F_{A,f} = \DA$ (Proposition \ref{prop:FAfDA}),
the isomorphisms classes of the family of cocycle algebras 
$\FAf$ for $f \in C(X_A,\Z)$ completely 
determine the  topological conjugacy class of 
the one-sided topological Markov shift $(X_A,\sigma_A)$.
Hence it seems to deserve to study the family of cocycle algebras 
$\FAf$ for $f \in C(X_A,\Z)$.
In this paper, we will investigate
the family of the $C^*$-subalgebras.

Let us denote by $\Zp$ the set of nonnegative integers.
It is well-known that the algebra $\OA$ is realized as the $C^*$-algebra 
$C^*(G_A)$ of an amenable \'etale groupoid $G_A$ defied by
\begin{align*}
G_A := \{ 
& (x, n, z) \in X_A\times \Z \times X_A \mid  \text{there exist } k,l \in \Zp \text{ such that } \\
& n = k-l,\,\,  \sigma_A^k(x) = \sigma_A^l(z) \}
\end{align*}
(cf. \cite{MatuiPLMS}, \cite{MatuiCrelle},  \cite{Renault2000}, \cite{Renault2}, \cite{Renault3}, etc.).
Put the unit space 
$G_A^{(0)} = \{(x, 0, x) \in G_A\mid  x \in X_A \}.$
Define the maps 
\begin{equation*}
s(x, n, z) = (z,0,z) \in G_A,\qquad
r(x, n, z) = (x,0,x) \in G_A. 
\end{equation*}
The product and inverse operation are defined by
\begin{equation*}
(x, n, z)\cdot  (z, m, w) = (x, n+m, w), \qquad
(x, n, z)^{-1} = (z, -n, x).
\end{equation*}
The unit space $G_A^{(0)}$ is naturally homeomorphic to the shift space
$X_A$.
The subgroupoid
\begin{equation*}
G_A^{\AF} =\{(x,0,z) \in G_A\mid x,z,\in X_A\}
\end{equation*}
is an AF-groupoid whose $C^*$-algebra is isomorphic to the standard AF-algebra
$\F_A$ (cf. \cite{CK}, \cite{PutnamAMS}).

For $f \in C(X_A,\Z)$ and $n \in \N$,
we define $f^n \in C(X_A,\Z)$ by setting 
\begin{equation} \label{eq:fnx}
f^n(x) =\sum_{i=0}^{n-1}f(\sigma_A^i(x)), \qquad x \in X_A.
\end{equation} 
 For $n=0$, we put $f^0\equiv 0$.
 It is straightforward to see that 
 the identity
$ f^{n+k}(x) = f^n(x) + f^k(\sigma_A^n(x)),\, x \in X_A, \, n,k \in \Zp
$
 holds.
\begin{definition}
For $f\in C(X_A,\Z)$,
define an  \'etale subgroupoid $G_{A,f}$ of $G_A$ by 
\begin{align*}
G_{A,f} := \{ &
(x, n, z) \in X_A\times \Z \times X_A \mid  \text{there exist } k,l \in \Zp \text{ such that }\\
&  n= k-l, \, \sigma_A^k(x) = \sigma_A^l(z), \,  
f^k(x) = f^l(z) \}.
\end{align*}
It is called the
{\it  cocycle groupoid for}\/ $f$. 
\end{definition}

We will see the following result.

\begin{theorem}[{Proposition \ref{prop:essprinamenable}, Theorem \ref{thm:main1}}]
The groupoid $G_{A, f}$ 
is an essentially principal amenable  clopen \'etale subgroupoid of $G_A$
such that 
there exists an isomorphism $\varPhi: C^*(G_A) \longrightarrow \OA$
of $C^*$-algebra such that 
$\varPhi( C^*(G_{A,f})) =\F_{A,f}$
and
$\varPhi(C(G_A^{(0)})) = \DA$.
\end{theorem}
Simplicity condition of the $C^*$-algebra $\F_{A,f}$
is obtained in Proposition \ref{prop:simplicity}.
We will then see that
there exists an isomorphism $\Phi: \OA\longrightarrow \OB$ 
of $C^*$-algebras such that 
$\Phi(\DA) = \DB$ and $\Phi(\F_{A,f}) = \F_{B,g}$ if and only if
there exists an isomorphism $\varphi: G_A\longrightarrow G_B$
of \'etale groupoids such that 
$\varphi(G_{A,f}) = G_{B,g}$
(Proposition \ref{prop:coeandcocycle}).
Therefore we will know the following characterization of 
topological conjugacy of one-sided topological Markov shifts 
in terms of these \'etale groupoids in the following way.
\begin{theorem}[{Corollary \ref{cor:conjugacygroupoids}}]
One-sided topological Markov shifts
 $(X_A,\sigma_A)$ and $(X_B, \sigma_B)$
 are topologically conjugate if and only if there exists an isomorphism
 $\varphi:G_A\longrightarrow G_B$ of \'etale groupoids
 such that
 $\varphi(G_{A, g\circ h}) = G_{B,g}$ for all 
 $g \in C(X_B,\Z),$ 
where $h:X_A \longrightarrow X_B$
is a homeomorphism defined by the restriction of $\varphi$
to its unit space $G_A^{(0)}$ under the identification between 
$G_A^{(0)}$ and $X_A$, and  $G_B^{(0)}$ and $X_B$, respectively. 
\end{theorem}

In the second half of the paper, we will 
study the following three subcalsses of cocycle algebras.
\begin{definition}\hspace{6cm}
\begin{enumerate}
\renewcommand{\theenumi}{\roman{enumi}}
\renewcommand{\labelenumi}{\textup{(\theenumi)}}
\item
For a subset $H \subset \{1,2,\dots, N\}$,
let 
 $\chi_H \in C(X_A,\Z)$ be the function defined by
\begin{equation}\label{eq:fHmu1}
\chi_H((x_n)_{n \in \N}) =
\begin{cases}
1 & \text{ if } x_1 \in H, \\
0 & \text{ otherwise.} 
\end{cases} 
\end{equation}
Then the function 
$\chi_H$ gives rise to an element of $C(X_A,\Z)$.
We set $\F_{A,H} := \F_{A,\chi_H}$ the cocycle algebra for $\chi_H$.
It is called the {\it support algebra for}\/ $H$.
\item
 For a function $b \in C(X_A, \Z)$, 
 put $1_b= 1 - b + b\circ\sigma_A \in C(X_A,\Z).$
We write
$\F_A^b := \F_{A, 1_b}$ the cocycle algebra for $1_b$.
It is called the {\it coboundary algebra for}\/ $b$.
\item
For a positive integer valued function $f \in C(X_A, \N)$, the cocycle algebra $\FAf$ is called the 
{\it suspension algebra}\/ for $f$.
\end{enumerate}
\end{definition}
 We will finally study isomorphism classes, stable isomorohism classes
of these cocycle algebras belonging to the above three classes
from different view points.  
A subset $H\subset \{1,2,\dots,N\}$ 
is said to be {\it saturated}\/ if any periodic word of $B_*(X_A)$
contains a symbol belonging to $H$.
  We thus finally obtain the following results.
\begin{theorem}[{Theorem \ref{thm:support}, Theorem \ref{thm:coboundary} 
and Theorem \ref{thm:suspension}}] 
Let $A$ be an irreducible non permutation matrix with entries in $\{0,1\}$.
 \hspace{4cm}
\begin{enumerate}
\renewcommand{\theenumi}{\roman{enumi}}
\renewcommand{\labelenumi}{\textup{(\theenumi)}}
\item
For a saturated subset $H\subset \{1,2,\dots,N\}$,
the support algebra $\F_{A,H}$ is a unital AF-algebra defined 
by a certain inclusion matrix $A_H$ defined in \eqref{eq:matrixAH}.
Furthermore, if $H\subset \{1,2,\dots,N\}$ is primitive
in the sense of Definition \ref{def:primitive},
the AF-algebra  $\F_{A,H}$ is simple.
\item Assume that $A$ is primitive. 
 For $b \in C(X_A,\Z)$, 
the coboundary algebra $\F_A^b$ is a unital simple AF-algebra stably isomorphic to 
the standard AF-algebra $\FA$ of $\OA$.
\item  Assume that $A$ is primitive.
For $f \in C(X_A,\N)$, 
the suspension algebra $\F_{A,f}$ is a unital simple AF-algebra stably isomorphic to 
the standard AF-algebra defined by the suspended matrix of the 
$K$-higher block matrix $A^{[K]}$ of $A$ 
 by the ceiling function $f$ for some $K$.
\end{enumerate}
\end{theorem}
 The proofs are given by three different ways for each.
 

 Throughout the paper,
 $S_1,\dots, S_N$ denote the canonical generating partial isometries 
 of $\OA$ satisfying 
 $\sum_{j=1}^N S_j S_j^* =1, \, S_i^* S_i = \sum_{j=1}^N A(i,j)S_j S_j^*, i=1,\dots, N$.
 For a word $\mu = (\mu_1,\dots,\mu_m) \in B_m(X_A)$, the partial isometry
$S_\mu $ is defined by $S_\mu =  S_{\mu_1}\cdots S_{\mu_m}$ in $\OA$.

\medskip

The contents of the present  paper is the following:

\medskip

1. Introduction and Prelimary

2. Cocycle algebras and cocycle groupoids

\hspace{5mm} 2.1. Basic properties of cocycle algebras

\hspace{5mm} 2.2. Basic properties of cocycle groupoids

\hspace{5mm} 2.3. The $C^*$-algebra $C^*(G_{A,f})$ 


\hspace{5mm} 2.4. Continuous orbit equivalence 

3. Three classes of cocycle algebras

\hspace{5mm}  3.1. Support algebras

\hspace{5mm}  3.2. Coboundary algebras

\hspace{5mm}  3.3. Suspension algebras


\section{Cocycle algebras and cocycle groupoids}
\subsection{Basic properties of cocycle algebras}
In this subsection, 
we present several basic lemmas to study cocycle algebras.
Take and fix an irreducible non permutation matrix $A$
and a continuous function $f \in C(X_A,\Z)$.
It is easy to see that 
$\DA \subset \FAf$ for any $f \in C(X_A,\Z)$.
For a word $\mu =(\mu_1,\dots,\mu_m) \in B_m(X_A)$,
let us denote by $|\mu|$ the length $m$ of $\mu$.
We first note the following lemma.
\begin{lemma}[{\cite[Lemma 3.1]{MaMZ}}] \label{lem:rhoafsmu}
For $\mu \in B_*(X_A)$, the  identity 
\begin{equation}
\rho^{A,f}_t(S_\mu) = \exp{(2\pi\sqrt{-1}f^{|\mu|} t)} S_\mu,\qquad
t \in \T  \label{eq:rhoafsmu}
\end{equation}
holds, where $f^{|\mu|}\in C(X_A, \Z)$ is defined by \eqref{eq:fnx}.
\end{lemma}
The following is basic to analyze the structure of the algebra $\FAf$. 
\begin{lemma} \label{lem:smusnufaf} 
For $\mu, \nu \in B_*(X_A)$, the partial isometry 
$S_\mu S_\nu^*$ belongs to $\FAf$ if and only if 
$f^{|\mu|} S_\mu S_\nu^* = S_\mu S_\nu^*  f^{|\nu|}$.
\end{lemma}
\begin{proof}
By \eqref{eq:rhoafsmu}, we know that
\begin{align*}
\rho^{A,f}_t(S_\mu S_\nu^*) 
=&  \exp{(2\pi\sqrt{-1}f^{|\mu|} t)} S_\mu S_\nu^* \exp{(-2\pi\sqrt{-1}f^{|\nu|} t)} \\
=&  \exp{(2\pi\sqrt{-1}f^{|\mu|} t)} S_\mu S_\nu^* \exp{(-2\pi\sqrt{-1}f^{|\nu|} t)} 
S_\nu S_\mu^* S_\mu S_\nu^* \\
=&  \exp{(2\pi\sqrt{-1}( f^{|\mu|}S_\mu S_\nu^* S_\nu S_\mu^*
         - S_\mu S_\nu^* f^{|\nu|} S_\nu S_\mu^*) t)} S_\mu S_\nu^*.
\end{align*}
Hence $S_\mu S_\nu^* \in \FAf$ if and only if
\begin{equation}
  f^{|\mu|}S_\mu S_\nu^* S_\nu S_\mu^*
         - S_\mu S_\nu^* f^{|\nu|} S_\nu S_\mu^* =0. \label{eq:exp2}
\end{equation}
The equality \eqref{eq:exp2} is equivalent to the equality
$f^{|\mu|} S_\mu S_\nu^* = S_\mu S_\nu^*  f^{|\nu|}$.
\end{proof}
For $b \in C(X_A, \Z)$, we define 
$1_b\in C(X_A, \Z)$ by
$1_b(x) = 1 - b(x) + b(\sigma_A(x)), x \in X_A$. 
\begin{lemma}\label{lem:smusnufab}
For 
$\mu, \nu \in B_*(X_A)$, the partial isometry 
$S_\mu S_\nu^*$ belongs to $\F_{A,1_b}$ 
if and only if 
$(|\mu| - b) S_\mu S_\nu^* = S_\mu S_\nu^* (|\nu| - b)$.
\end{lemma}
\begin{proof}
By Lemma \ref{lem:smusnufaf}, we know that 
 $S_\mu S_\nu^* \in \FAf$ if and only if 
$f^{|\mu|} S_\mu S_\nu^* = S_\mu S_\nu^*  f^{|\nu|}$.
Put $m = |\mu|, n=|\nu|$.
Now $ f= 1 -b + b\circ\sigma_A$ so that 
$f^m = m -b + b\circ \sigma_A^m$.  
Hence the equality
$f^{|\mu|} S_\mu S_\nu^* = S_\mu S_\nu^*  f^{|\nu|}$
is equivalent to the equality
\begin{equation*}
(m - b +b\circ\sigma_A^m) \, S_\mu S_\nu^*
=S_\mu S_\nu^* \, (n - b +b\circ\sigma_A^n). 
\end{equation*}
Under the identification between $C(X_A)$ and $\DA$, 
we see that 
$b\circ \sigma_A^m = \sum_{\xi\in B_m(X_A)}S_\xi b S_\xi^*$,
so that 
$$
 b\circ \sigma_A^m \, S_\mu S_\nu^*
= 
\sum_{\xi\in B_m(X_A)}S_\xi b S_\xi^* S_\mu S_\nu^*
= S_\mu  b S_\mu^* S_\mu S_\nu^*
= S_\mu b S_\nu^*
$$
and similarly
$
S_\mu S_\nu^*  \, b\circ\sigma_A^n 
=S_\mu b S_\nu^*.
$
Hence we have 
$
 b\circ \sigma_A^m \, S_\mu S_\nu^*
 =S_\mu S_\nu^*  \, b\circ\sigma_A^n.
$
Therefore
$f^{|\mu|} S_\mu S_\nu^* = S_\mu S_\nu^*  f^{|\nu|}$
is equivalent to the equality
\begin{equation}
(m-b) S_\mu S_\nu^* = S_\mu S_\nu^*(n-b). \label{eq:mbsmunu}
\end{equation}
\end{proof}
We note that  if in particular $m=n$,
 the equality \eqref{eq:mbsmunu} goes to
\begin{equation*}
b \, S_\mu S_\nu^* = S_\mu S_\nu^* \, b, 
\end{equation*}
that means 
$S_\mu S_\nu^* \in \{ b\}^\prime \cap \FA$
for $S_\mu S_\nu^* \in \F_{A,1_b}$ with $|\mu| = |\nu|$. 

\begin{proposition}\label{prop:FAfDA}
$$\bigcap_{f \in C(X_A,\Z)} \FAf =\bigcap_{b \in C(X_A,\Z)} \F_{A,1_b} = \DA.$$
\end{proposition}
\begin{proof}
As $\FAf \supset \DA$ for all $f \in C(X_A,\Z)$,
we will show the inclusion $\bigcap_{b \in C(X_A,\Z)} \F_{A,1_b} \subset \DA.$
By the preceding lemma, we see that 
\begin{equation*}
\F_{A,1_b} \cap \FA = \{ X \in \FA \mid b X =Xb\}. 
\end{equation*}
For $b\equiv 0$, we have 
$\F_{A,1_b} = \FA$.
For $\mu \in B_*(X_A)$, let $b =\chi_{U_\mu}$ that is identified with
$S_\mu S_\mu^*$. 
Hence we have
$\{ b\}^\prime \cap \FA = \{ S_\mu S_\mu^*\}^\prime \cap \FA,
$ 
Since the projections $S_\mu S_\mu^*, \mu \in B_*(X_A)$
generate $\DA$, we know that 
$$
\bigcap_{\mu \in B_*(X_A)} \{ S_\mu S_\mu^*\}^\prime \cap \FA
=\DA^\prime \cap \FA = \DA,
$$
because $\DA$ is maximal abelian in $\FA$.
We conclude that 
$\bigcap_{b \in C(X_A,\Z)} \F_{A,1_b} = \DA$
and hence 
$\bigcap_{f \in C(X_A,\Z)} \FAf = \DA.$
\end{proof}

\begin{lemma}
For $\mu, \nu \in B_*(X_A)$, the partial isometry 
$S_\mu S_\nu^*$ belongs to $\FAf$
 if and only if 
$S_{\mu i} S_{\nu i}^*$ belongs to $\FAf$
for all $i \in \{1,2,\dots,N\}$ satisfying 
$\mu i, \, \nu i \in B_*(X_A)$.
\end{lemma}
\begin{proof}
Since the identity
$S_\mu S_\nu^* = \sum_{i=1}^N S_{\mu i} S_{\nu i}^*$
holds,
it suffices to prove the only if part of the assertion.
For $\mu, \nu \in B_*(X_A)$, put
$m=|\mu|, n =|\nu|$.
Suppose that the partial isometry 
$S_\mu S_\nu^*$ belongs to $\FAf$.
By Lemma \ref{lem:smusnufaf}, 
the equality
$f^{m} S_\mu S_\nu^* = S_\mu S_\nu^*  f^n$
holds.
Since 
$S_\mu S_\nu^* = \sum_{i=1}^N S_{\mu i} S_{\nu i}^*$,
we know that 
\begin{equation*}
 S_\mu S_\nu^*\cdot S_{\nu i} S_{\nu i}^* 
=S_{\mu i} S_{\nu i}^*S_{\nu i} S_{\nu i}^*
=S_{\mu i} S_{\nu i}^*. 
\end{equation*} 
As $f^n$ commutes with $S_{\nu i} S_{\nu i}^*$,
we see the equality
\begin{equation}
f^{m}  S_{\mu i} S_{\nu i}^* 
=f^{m} S_\mu S_\nu^*\cdot S_{\nu i} S_{\nu i}^* 
= S_\mu S_\nu^*  f^n \cdot S_{\nu i} S_{\nu i}^*
= S_{\mu i} S_{\nu i}^* f^n. \label{eq:smuinui4}
\end{equation}
We also have
\begin{equation*}
f\circ \sigma_A^m \cdot  S_{\mu i} S_{\nu i}^*
= \sum_{\xi \in B_m(X_A)} S_\xi f S_\xi^* S_{\mu i} S_{\nu i}^* 
= S_\mu f S_\mu^* S_{\mu} S_i S_i^* S_{\nu}^* 
= S_\mu f  S_i S_i^* S_{\nu}^* 
\end{equation*}
and
\begin{equation*}
 S_{\mu i} S_{\nu i}^*\cdot f\circ \sigma_A^n 
= S_{\mu i} S_{\nu i}^* \sum_{\eta \in B_n(X_A)} S_\eta f S_\eta^* 
= S_{\mu} S_i S_i^* S_{\nu}^* S_\nu f S_\nu^*  
= S_\mu f  S_i S_i^* S_{\nu}^* 
\end{equation*}
so that we have
\begin{equation}
f\circ \sigma_A^m \cdot  S_{\mu i} S_{\nu i}^*
=S_{\mu i} S_{\nu i}^*\cdot f\circ \sigma_A^n. \label{eq:smuinui6}
\end{equation}
As $f^{m+1} = f^m + f\circ \sigma_A^m$
and 
$f^{n+1} = f^n + f \circ \sigma_A^n$,
we have the equality
\begin{equation*}
f^{m+1} \cdot  S_{\mu i} S_{\nu i}^*
=S_{\mu i} S_{\nu i}^*\cdot f^{n+1}
\end{equation*}
by \eqref{eq:smuinui4} and \eqref{eq:smuinui6}.
Hence
 $S_{\mu i} S_{\nu i}^*$ belongs to $\FAf$.
\end{proof}
Define an expectation
$E_{A,f}: \OA \longrightarrow \FAf$ by 
\begin{equation*}
E_{A,f}(X ) = \int_{\T} \rho^{A,f}_t(X) \, dt \qquad \text{ for } X \in \OA,
\end{equation*}
where $dt$ stands for the normalized Lebesgue measure on $\T$.
Let $\PA$ be the $*$-algebra algebraically generated by $S_1, \dots, S_N$.
For $S_\mu S_\nu^* \in \PA$, we have
\begin{align*}
E_{A,f}(S_\mu S_\nu^*)
& = \int_{\T} \exp{(2\pi\sqrt{-1} f^{|\mu|} t)} S_\mu S_\nu^* \exp{(-2\pi\sqrt{-1} f^{|\nu|} t)}\, dt \\ 
& = \int_{\T} \exp{(2\pi\sqrt{-1} 
( f^{|\mu|} - S_\mu S_\nu^* f^{|\nu|} S_\nu S_\mu^*)  t)} \, dt \,
S_\mu S_\nu^*. 
\end{align*}
Put
$f_{\mu,\nu} = f^{|\mu|} - S_\mu S_\nu^* f^{|\nu|} S_\nu S_\mu^* \in \DA$
and
$$
P_{\mu,\nu} = \int_{\T} \exp{(2\pi\sqrt{-1}  f_{\mu,\nu} t)} \, dt \in \DA,
$$
that is a projection in $\DA$.
We thus obtain the following lemma. 
\begin{lemma} \label{lem:2.6}
$E_{A,f}(S_\mu S_\nu^*) = P_{\mu,\nu}S_\mu S_\nu^*$. 
\end{lemma}
Hence we have 
\begin{lemma}
$\PA\cap \FAf$ is dense in $\FAf$.
\end{lemma}
\begin{proof}
As Lemma \ref{lem:2.6},
the equality
$E_{A,f}(S_\mu S_\nu^*) = P_{\mu,\nu}S_\mu S_\nu^*$
holds.
Since 
$P_{\mu,\nu} $ is a projection in $\DA,$
it belongs to $\PA\cap \FAf$,
so that 
$E_{A,f}(S_\mu S_\nu^*) \in \PA\cap \FAf$.
This means 
$E_{A,f}(\PA) \subset  \PA\cap \FAf$
and hence 
$E_{A,f}(\PA) = \PA\cap \FAf.$
To show that $\PA\cap \FAf$ is dense in $\FAf$,
it suffices to prove that 
$E_{A,f}(\PA)$ is dense in $\FAf$.
Take an arbitrary $X \in \FAf$.
Since $\PA$ is dense in $\OA$,
there exists a sequence $X_n\in \PA$ such that 
$\lim_{n\to{\infty}}\| X_n - X \| =0.$
As we have 
\begin{equation*}
\| E_{A,f}(X_n) - X\| =
\| E_{A,f}(X_n - X) \| \le
\| X_n - X \| 
\end{equation*}
and
$E_{A,f}(X_n)$ belongs to $\PA \cap \FAf$,
we know that 
$\PA\cap \FAf$ is dense in $\FAf$.
\end{proof}

Therefore we have the following proposition.

\begin{proposition}\label{prop:generated}
The $C^*$-algebra $\FAf$ is generated by the partial isometries of the form
\begin{equation*}
S_\mu S_\nu^* \quad \text{ satisfying } \quad
f^{|\mu|}S_\mu S_\nu^* = S_\mu S_\nu^*f^{|\nu|},\quad 
\mu, \nu \in B_*(X_A). 
\end{equation*}
\end{proposition}

\subsection{Basic properties of cocycle groupoids}
In this subsection,
 we will study a family $G_{A,f}, f \in C(X_A,\Z)$  of 
\'etale subgroupoids of $G_A$.
They are called the {\it cocycle groupoid for}\/ $f$.
For $f\in C(X_A,\Z)$,
define the \'etale subgroupoid $G_{A,f}$ by setting
\begin{align*}
G_{A,f} := \{ &
(x, n, z) \in X_A\times \Z \times X_A \mid  \text{there exist } k,l \in \Zp \text{ such that }\\
& n= k-l, \, \sigma_A^k(x) = \sigma_A^l(z), \,   
f^k(x) = f^l(z) \}.
\end{align*}
Put the unit space 
$G_{A,f}^{(0)} = \{(x, 0, x) \in G_{A,f}\mid  x \in X_A \}.$
The product and the inverse operation on $G_{A,f}$ are
inherited from $G_A$.
\begin{lemma}
\begin{enumerate}
\renewcommand{\theenumi}{\roman{enumi}}
\renewcommand{\labelenumi}{\textup{(\theenumi)}}
\item 
If $(x,n,z), (z, m,w) \in G_{A,f}$, then 
$(x, n+m, w) \in G_{A,f}$.
\item 
If $(x,n,z) \in G_{A,f}$, then 
$(z, -n, x) \in G_{A,f}$.
\end{enumerate}
\end{lemma}
\begin{proof}
For $(x,n,z), (z, m,w) \in G_{A,f}$,
take $k,l, p,q \in \Zp$ such that 
\begin{gather*}
n= k-l, \qquad \sigma_A^k(x) = \sigma_A^l(z), \qquad   
f^k(x) = f^l(z), \\
m= p-q, \qquad \sigma_A^p(z) = \sigma_A^q(w), \qquad   
f^p(z) = f^q(w).
\end{gather*}
We then have 
\begin{equation*}
\sigma_A^{p+k}(x) 
= \sigma_A^p(\sigma_A^k(x))
= \sigma_A^p(\sigma_A^l(z))
= \sigma_A^l(\sigma_A^p(z))
= \sigma_A^l(\sigma_A^q(w))
= \sigma_A^{l+q}(w).
\end{equation*}
Since $\sigma_A^k(x) = \sigma_A^l(z)$
and
$\sigma_A^p(z) = \sigma_A^q(w)$, we have 
\begin{align*}
f^{p+k}(x) 
=&  f^k(x)  + f^p(\sigma_A^k(x)) 
= f^l(z)  + f^p(\sigma_A^l(z)) 
=  f^{l+p}(z) \\
=&  f^p(z)  + f^l(\sigma_A^p(z)) 
=  f^q(w)  + f^l(\sigma_A^q(w)) 
= f^{l+q}(w).
\end{align*}

The assertion (ii) is obvious. 
\end{proof}
Recall that the topology on the groupoid $G_A$ is defined in the following way.
For open subsets $U, V \subset X_A$
and $k,l \in \Zp$ such that both
$\sigma_A^k: U \longrightarrow \sigma_A^k(U)$
and
$\sigma_A^l: V \longrightarrow \sigma_A^l(V)$
are injective and hence homeomorphisms,
an open neighborhood basis of $G_A$ is defined by
\begin{equation*}
\mathcal{ U}(U, k,l,V) 
=  \{ ( y, k-l, w) \in G_A 
         \mid y \in U, w \in V, \sigma_A^k(y ) = \sigma_A^l(w)\}
\end{equation*}
As $G_{A,f} \subset G_A$, 
we endow $G_{A,f}$ with the relative topology
from $G_A$, so that 
$G_{A,f}$ is a topological subgroupoid of $G_A$ such that 
the unit space 
$G_{A,f}^{(0)}$ is homeomorphic to $X_A$ 
under the natural correspondence
$(x,0,x) \in G_{A,f}^{(0)} \longrightarrow x \in X_A$.
We henceforth identify $G_{A,f}^{(0)}$ with $X_A$
through the natural identification.

\begin{lemma}
$G_{A,f}$ is open and closed in $G_A$.
Hence  $G_{A,f}$ is a clopen subgroupoid of $G_A$.
\end{lemma}
\begin{proof}
For $(x,n,z) \in G_{A,f},$
there exist $k,l \in \Zp$  
such that 
$n = k-l, \, \sigma_A^k(x) = \sigma_A^l(z)$
and $f^k(x) = f^l(z).$
Both $f^k$ and $f^l$ are continuous
so that 
there exist $p, q \in \N$ with $p\ge k, \, q\ge l$
such that 
\begin{equation} \label{eq:fkxp}
f^k(x) = f^k(x') \text{ for all } x' \in U_{(x_1,\dots, x_p)} \quad \text{ and } \quad
f^l(z) = f^l(z') \text{ for all } z' \in U_{(z_1,\dots, z_q)}. 
\end{equation}
Hence we have  
$$
(x, n, z) \in 
\mathcal{U}(U_{(x_1,\dots, x_p)}, k,l,
U_{(z_1,\dots, z_q)})
$$
and
$$
\sigma_A^k(x') = \sigma_A^l(z') \quad
\text{ and } \quad
f^k(x') = f^l(z')
$$
for all 
$
(x', n, z') \in 
\mathcal{U}(U_{(x_1,\dots, x_p)}, k,l,
U_{(z_1,\dots, z_q)}).
$
This means that 
$\mathcal{U}(U_{(x_1,\dots, x_p)}, k,l,
U_{(z_1,\dots, z_q)})$
is an open neighborhood of $(x, k-l,z)$
in $G_{A,f}$,
proving that $G_{A,f}$ is open in $G_A$.

We will next show that 
$G_{A,f}$ is closed in $G_A$.
Let $g_\alpha=(x_\alpha, n_\alpha, z_\alpha) \in G_{A,f}, \alpha \in \Lambda$
be a net converging to $g =(x,n,z) \in G_A$. 
Take $k,l\in \Zp$ such that 
$n = k-l, \sigma_A^k(x) = \sigma_A^l(z)$.
Since  both $f^k, f^l$ are continuous on $X_A$
one may find  $p, q \in \N$ with $p\ge k, \, q\ge l$
satisfying  
\eqref{eq:fkxp}, respectively.
As
$(x,n,z) \in 
\mathcal{U}(U_{(x_1,\dots, x_p)}, k,l,U_{(z_1,\dots, z_q)}),$
there exists $\alpha\in \Lambda$ such that 
$g_\alpha 
\in 
\mathcal{U}(U_{(x_1,\dots, x_p)}, k,l,U_{(z_1,\dots, z_q)}),$
so that 
\begin{equation}\label{eq:3.2.3}
n_\alpha = k-l, \quad
x_\alpha\in U_{(x_1,\dots, x_p)}, \quad
z_\alpha\in U_{(z_1,\dots, z_q)}, \quad
\sigma_A^k(x_\alpha) = \sigma_A^l(z_\alpha)
\end{equation}
and hence $n_\alpha = n$.
By 
$x_\alpha\in U_{(x_1,\dots, x_p)},
z_\alpha\in U_{(z_1,\dots, z_q)}$
as in \eqref{eq:3.2.3},
we have by \eqref{eq:fkxp}
\begin{equation}\label{eq:3.2.5}
f^k(x) =f^k(x_\alpha), \qquad f^l(z) =f^l(z_\alpha).
\end{equation}
On the other hand, 
as $g_\alpha =(x_\alpha, n_\alpha, z_\alpha) \in G_{A,f},$
there exist $k_\alpha, l_\alpha \in \Zp$ such that 
\begin{equation}\label{eq:3.2.4}
n_\alpha = k_\alpha -l_\alpha, \quad
\sigma_A^{k_\alpha}(x_\alpha) = \sigma_A^{l_\alpha}(z_\alpha), \quad
f^{k_\alpha}(x_\alpha) =f^{l_\alpha}(z_\alpha). 
\end{equation}
We note that for any $p_\alpha >k_\alpha$ and $q_\alpha >l_\alpha$
such that 
$p_\alpha - q_\alpha = n_\alpha
$
we have
\begin{equation*}
\sigma_A^{p_\alpha}(x_\alpha) = \sigma_A^{q_\alpha}(z_\alpha), \quad
f^{p_\alpha}(x_\alpha) =f^{q_\alpha}(z_\alpha). 
\end{equation*}
We indeed see that 
\begin{gather*}
\sigma_A^{p_\alpha}(x_\alpha) =
\sigma_A^{p_\alpha-k_\alpha}(\sigma_A^{k_\alpha}(x_\alpha)) =
\sigma_A^{q_\alpha-l_\alpha}(\sigma_A^{l_\alpha}(z_\alpha)) =
\sigma_A^{q_\alpha}(z_\alpha), \\
f^{p_\alpha}(x_\alpha) 
=f^{k_\alpha}(x_\alpha) +f^{p_\alpha-k_\alpha}(\sigma_A^{k_\alpha}(x_\alpha)) 
=f^{l_\alpha}(z_\alpha) +f^{q_\alpha-l_\alpha}(\sigma_A^{l_\alpha}(z_\alpha)) 
=f^{q_\alpha}(z_\alpha).
\end{gather*}
Hence we may assume that $k_\alpha > k$ and $l_\alpha >l$.
As we have
\begin{equation*}
f^{k_\alpha}(x_\alpha) 
 =f^k(x_\alpha) +f^{k_\alpha-k}(\sigma_A^k(x_\alpha)), \qquad  
f^{l_\alpha}(z_\alpha) 
 =f^l(z_\alpha) +f^{l_\alpha-l}(\sigma_A^l(z_\alpha))  
\end{equation*}
together with
 $f^{k_\alpha}(x_\alpha)=f^{l_\alpha}(z_\alpha)$,
$\sigma_A^k(x_\alpha) = \sigma_A^l(z_\alpha)$
and
$k_\alpha - k = l_\alpha -l$
by \eqref{eq:3.2.3} and \eqref{eq:3.2.4},
we see that 
\begin{equation}\label{eq:3.2.6}
f^k(x_\alpha)= f^l(z_\alpha).
\end{equation}
By \eqref{eq:3.2.5} and \eqref{eq:3.2.6},
we conclude 
$f^k(x)= f^l(z)$
so that 
$g = (x,n,z) $ belongs to $G_{A,f}$,
proving $G_{A,f}$ is closed in $G_A$.
\end{proof}
We note that if $f\equiv 0$, then $G_{A,f} = G_A$,
and
if $f\equiv 1$, then $G_{A,f} = G_A^{\AF}.$
Hence the \'etale groupoid $G_{A,f}$
is a generalization of both $G_A$ and $G_A^\AF$.
\begin{lemma}
Keep the above notation.
\begin{enumerate}
\renewcommand{\theenumi}{\roman{enumi}}
\renewcommand{\labelenumi}{\textup{(\theenumi)}}
\item $G_{A,f}$ is an \'etale groupoid.
\item $G_{A,f}$ is essentially principal.
\item $G_{A,f}$ is amenable.
\end{enumerate}
\end{lemma}
\begin{proof}
(i) It suffices to show that 
the map $r : G_{A,f} \longrightarrow G_{A,f}$ 
is a local homeomorphism.
Take 
$(x,n,z) \in G_{A,f}$.
There exists $k,l \in \Zp$ such that 
$\sigma_A^k(x) = \sigma_A^l(z)$
and
$n= k-l, 
f^k(x)= f^l(z).
$
Consider the following open neighborhood $\mathcal{U}$ 
of $(x,n,z)$ in $G_{A,f}$
\begin{align*}
\mathcal{U} 
= & \mathcal{U}(U_{(x_1,\dots,x_k)}, k,l, U_{(z_1,\dots,z_l)} ) \\
= &\{ (y, k-l, w) \in G_{A,f} \mid
 y \in U_{(x_1,\dots,x_k)}, w \in U_{(z_1,\dots,z_l)},
\sigma_A^k(y) = \sigma_A^l(w), f^k(y) = f^l(w)\}.
\end{align*}
It is straightforward to see that 
$r: \mathcal{U}\longrightarrow r(U)$ is injective
and hence homeomorphic by definition of the topology on $G_{A,f}$,
so that  the groupoid
$G_{A,f}$ is an \'etale groupoid.

(ii)
Put the isotropy bundles
$G^\prime_A = \{ (x, n, x) \in G_A \mid x \in X_A\}
$ and
\begin{equation*}
G^\prime_{A,f} = \{ (x, n, x) \in G_{A,f} \mid x \in X_A\},
 \end{equation*}
so that 
$G_{A,f}^{(0)} \subset G^\prime_{A,f} \subset G_A^\prime$. 
To prove that $G_{A,f}$ is essentially principal, 
we will show that the set 
$\Int(G^\prime_{A,f})$ of interiers of $G^\prime_{A,f}$ 
coincides with $G_{A,f}^{(0)}$.
As 
$G_{A,f}$ is \'etale, 
the unit space 
$G_{A,f}^{(0)}$ is open in $G_{A,f}$. 
Since $G_A$ is essentially principal,
$
\Int(G_A^{\prime}) = G_A^{(0)}  
$
so that we have
$$
G_{A,f}^{(0)} =\Int(G_{A,f}^{(0)}) \subset \Int(G^\prime_{A,f})
\subset \Int(G^\prime_A) 
= G_A^{(0)} = G_{A,f}^{(0)},
$$
proving  that 
$
G_{A,f}$
is essentially principal.

(iii)
Since $G_A$ is \'etale and amenable
together with $G_{A,f}$ being clopen in $G_A$,
we see that $G_{A,f}$ is amenable 
(see for example \cite[Proposition 4.1.14]{Sims}).
\end{proof}
Hence we have the following proposition.
\begin{proposition}\label{prop:essprinamenable}
Let $A$ be an irreducible, non permutation matrix with entries in $\{0,1\}$. 
For any $f \in C(X_A,\Z)$,
the cocycle groupoid 
 $G_{A,f}$ is an essentially principal amenable clopen \'etale subgroupoid of $G_A$.
\end{proposition}

\subsection{The $C^*$-algebra $C^*(G_{A,f})$ }

 In this subsection we will study 
 the groupoid $C^*$-algebra $C^*(G_{A,f})$.
Let $C_c(G)$ stand for the $*$-algebra of complex valued 
compactly supported continuous functions on $G$.
Its product structure and $*$-involution in $C_c(G)$ 
are defined by  
\begin{gather*}
(\xi*\eta)(t) = \sum_{s\in G, r(s) = r(t)} \xi(s) \eta(s^{-1}t),
\qquad
\xi^*(t) = \overline{\xi(t^{-1})}
\end{gather*} 
for $\xi, \eta \in C_c(G), t \in G$.
Define the $C^*$-algebra $C^*(G)$
by the completion of $C_c(G)$
by the universal $C^*$-norm on $C_c(G)$.
For the amenable \'etale groupoid $G$,
the $C^*$-algebra coincides with 
the reduced $C^*$-algebra $C_r^*(G)$ of $G$
(cf. 
\cite{Renault}, \cite{Sims}).
For the groupoid $G_A$, it is well-known that
there exists an isomorphism $\Phi: C^*(G_A) \longrightarrow \OA$
of $C^*$-algebras such that $\Phi(C(G_A^{(0)})) = \DA.$
\begin{lemma}\label{lem:embedHA}
Let $f \in C(X_A,\Z)$.
There exists a natural embedding 
$\iota_f:C^*(G_{A,f}) \hookrightarrow  C^*(G_A)$
such that 
$\iota_f(C(G_{A,f}^{(0)})) = C(G_A^{(0)}).$
\end{lemma}
\begin{proof}
Since $G_{A,f}$ is  clopen in $G_A$,
the $*$-algebra $C_c(G_{A,f})$ is naturally embedded  
into $C_c(G_A)$.
By the universality of the $C^*$-norm on $C^*(G_{A,f}),$
there exists a $*$-homomorphism
$\iota_f:C^*(G_{A,f}) \longrightarrow C^*(G_A).$
As it is identical on $C(G_{A,f}^{(0)})$,
the 
$*$-homomorphism
$\iota_f:C^*(G_{A,f}) \longrightarrow C^*(G_A)$
is actually injective
(cf. \cite[Theorem 4.2.9]{Sims}).
Hence the natural embedding $C_c(G_{A,f}) \hookrightarrow C_c(G_A)$
induces the embedding $\iota_f:C^*(G_{A,f}) \hookrightarrow C^*(G_A)$
as $C^*$-subalgebra such that 
$\iota_f(C(G_{A,f}^{(0)})) = C(G_A^{(0)}).$
\end{proof}
We will show the following theorem.
\begin{theorem}\label{thm:main1}
Assume that $A$ is an irreducible, non permutation matrix with entries in $\{0,1\}$.
Let $f \in C(X_A,\Z)$.
Then there exsits an isomorphism $\varPhi: C^*(G_A) \longrightarrow \OA$
of $C^*$-algebra such that 
$\varPhi( C^*(G_{A,f})) =\F_{A,f}$
and
$\varPhi( C(G_{A,f}^{(0)})) =\DA$. 
\end{theorem}
\begin{proof}
By Lemma \ref{lem:embedHA}, 
we may regard $C^*(G_{A,f})$ 
as a $C^*$-subalgebra of $C^*(G_A).$
For $\mu, \nu \in B_*(X_A)$ with
$|\mu| =k,\, |\nu| = l$, put
\begin{equation*}
U_{\mu,\nu} = \{(x,k-l,y) \in G_A \mid 
x \in U_\mu, y \in U_\nu, \, \sigma_A^k(x)=\sigma_A^l(y)\}
\end{equation*}
that is a clopen set in $G_A$.
Let $\chi_{U_{\mu,\nu}}$ 
be the characteristic function of $U_{\mu,\nu}$ on $G_A$.
It is well-known that 
the correspondence 
$\Phi: \chi_{U_{\mu,\nu}}\in C^*(G_A)\longrightarrow S_\mu S_\nu^* \in \OA$
gives rise to an isomorphism
from $C^*(G_A)$ onto $\OA$.
Put
\begin{equation*}
V_{\mu,\nu} = \{(x,|\mu| - |\nu|,z) \in G_A \mid 
x \in U_\mu, z \in U_\nu, \, \sigma_A^{|\mu|}(x)=\sigma_A^{|\nu|}(z), \,
f^{|\mu|}(x) = f^{\nu|}(z) \}
\end{equation*}
By noticing that the continuous functions 
$f^{|\mu|}, f^{|\nu|}$ on $X_A$ are regarded as 
elements of $C_c(G_{A,f})$ through the natural identification 
between $X_A$ an $G_{A,f}^{(0)}$ so that  
\begin{equation*}
f^{|\mu|}(x, n,z) =
\begin{cases}
f^{|\mu|}(x) & \text{ if } n=0, \, z=x, \\
0 & \text{ oterwise, }
\end{cases}
\qquad
f^{|\nu|}(x, n,z) =
\begin{cases}
f^{|\nu|}(x) & \text{ if } n=0, \, z=x, \\
0 & \text{ oterwise, }
\end{cases}
\end{equation*}
we have
\begin{align*}
(f^{|\mu|} * \chi_{V_{\mu,\nu}})(x,n,z)
 =&  f^{|\mu|}(x) \chi_{V_{\mu,\nu}}(x,n,z), \\
( \chi_{V_{\mu,\nu}}*f^{|\nu|} )(x,n,z)
 =& \chi_{V_{\mu,\nu}}(x,n,z) f^{|\nu|}(z).
\end{align*}
Hence we have
\begin{equation*}
f^{|\mu|} * \chi_{V_{\mu,\nu}}
=
\chi_{V_{\mu,\nu}}*f^{|\nu|}.
\end{equation*}
The $C^*$-algebra $C^*(G_{A,f})$ 
is generated by the characteristic functions
$\chi_{V_{\mu,\nu}}$
of $V_{\mu,\nu}, \mu, \nu \in B_*(X_A)$.
By using Proposition \ref{prop:generated}
together with Lemma \ref{lem:embedHA},
we have
$\Phi:C^*(G_A) \longrightarrow \OA$
satisfies
$\varPhi( C^*(G_{A,f})) =\F_{A,f}$
and
$\varPhi( C(G_{A,f}^{(0)})) =\DA$. 
\end{proof}


We will state a simplicity condition
of the $C^*$-algebra $C^*(G_{A,f})$
in terms of the function $f$
in the following way.
Following \cite{MatuiPLMS},
 an \'etale groupoid $G$ is said to be {\it minimal}\/
 if for any $z\in G^{(0)}$, 
 the set $\{r(g) \in G^{(0)} \mid s(g) =z\}$
 is dense in $G^{(0)}.$  
\begin{definition}\label{def:fminimal}
A function $f \in C(X_A,\Z)$ is said to be {\it minimal}\/ if
for $z \in X_A$ and $\mu \in B_*(X_A)$, there exist
$x \in U_\mu$ and $k, l \in \Zp$ such that 
\begin{equation}\label{eq:defsimplef}
\sigma_A^k(x) = \sigma_A^l(z), \qquad
f^k(x) = f^l(z).
\end{equation}
\end{definition}
By the above definition of the minimality of the function $f$,
the next lemma is immediate. 
\begin{lemma}\label{lem:fminimal}
A function $f\in C(X_A,\Z)$ is minimal if and only if the groupoid 
$G_{A,f}$ is minimal.
\end{lemma}

We note that the \'etale groupoid $G_{A,f}$ is essentially principal and amenable.
By \cite{MatuiPLMS}, 
an essentially principal amenable \'etale groupoid $G$ is minimal 
if and only if the $C^*$-algebra $C^*(G)$ of the groupoid is simple.
Therefore we have the following proposition.
\begin{proposition}\label{prop:simplicity}
A function $f\in (X_A,\Z)$ is minimal if and only if the $C^*$-algebra $\FAf$ is simple.
\end{proposition}

\subsection{Continuous orbit equivalence}

Let $A, B$ be irreducible non permutation matrices with entries in $\{0,1\}$.
It is well-known that 
the groupoids $G_A$ and $G_B$  
are isomorphic as \'etale groupoids if and only if
there exists an isomorphism $\Phi: \OA\longrightarrow \OB$ of $C^*$-algebras such that 
$\Phi(\DA) = \DB$ 
(cf. \cite{MMKyoto},  \cite{MatuiPLMS}, \cite{MatuiCrelle}, \cite{Renault2}, \cite{Renault3}).
In \cite{MaPacific}, 
a notion called continuous orbit equivalence generalizing topological conjugacy in one-sided 
topological Markov shifts was introduced. 
One-sided topological Mrkov shifts
$(X_A,\sigma_A)$ and $(X_B,\sigma_B)$ are said to be 
{\it continuous orbit equivalent}\/ if there exist $k_1, l_1 \in C(X_A,\Zp)$
and $k_2, l_2 \in C(X_B,\Zp)$ such that 
$\sigma_B^{k_1(x)}(h(\sigma_A(x))) = \sigma_B^{l_1(x)}(h(x)), x \in X_A$
and
$\sigma_A^{k_2(y)}(h^{-1}(\sigma_B(y))) = \sigma_A^{l_2(y)}(h^{-1}(y)), y \in X_B.$
As in \cite{MaPacific}, 
it was shown that 
$(X_A,\sigma_A)$ and $(X_B,\sigma_B)$ are continuous orbit equivalent 
if and only if  
there exists an isomorphism $\Phi: \OA\longrightarrow \OB$ 
of $C^*$-algebras such that 
$\Phi(\DA) = \DB$. 
Hence 
one-sided topological Mrkov shifts
$(X_A,\sigma_A)$ and $(X_B,\sigma_B)$ are  
 continuous orbit equivalent
 if and only if 
the \'etale groupoids $G_A$ and $G_B$ are isomorphic
(\cite{MMKyoto}).
In our situation, we see the following proposition.
\begin{proposition}\label{prop:coeandcocycle}
Let $f \in C(X_A, \Z)$ and $g \in C(X_B,\Z)$.
There exists an isomorphism $\Phi: \OA\longrightarrow \OB$ of $C^*$-algebras such that 
$\Phi(\DA) = \DB$ and $\Phi(\F_{A,f}) = \F_{B,g}$ if and only if
there exists an isomorphism $\varphi: G_A\longrightarrow G_B$
of \'etale groupoids such that 
$\varphi(G_{A,f}) = G_{B,g}$.
\end{proposition}
\begin{proof}
We note that 
$\F_{A,f}= C^*(G_{A,f})$ and 
$\F_{B,g} = C^*(G_{B,g})$.
Hence by the proof of 
\cite[Proposition 4.1]{Renault2} 
(see also \cite[Theorem 5.17]{MatuiPLMS}),
we know the desired assertion.
We in fact see that the condition that 
there exists an isomorphism $\varphi: G_A\longrightarrow G_B$
of \'etale groupoids such that 
$\varphi(G_{A,f}) = G_{B,g}$
implies an isomorphism $\Phi: \OA\longrightarrow \OB$ of $C^*$-algebras such that 
$\Phi(\DA) = \DB$ and $\Phi(\F_{A,f}) = \F_{B,g}$.
Conversely,
suppose that there exists an isomorphism $\Phi: \OA\longrightarrow \OB$ 
of $C^*$-algebras such that 
$\Phi(\DA) = \DB$ and $\Phi(\F_{A,f}) = \F_{B,g}$.
By considering of germs of the normalizers of the subalgebras $\DA\subset \OA$
and $\DB \subset \OB$, 
we know that the additional condition 
$\Phi(\FAf) = \F_{B,g}$ implies the restriction of the isomorphism
$\varphi: G_A \longrightarrow G_B$ of \'etale groupoids yields 
the equality $\varphi(G_{A,f}) = G_{B,g}$.
\end{proof}

\begin{corollary}\label{cor:3.11}
Suppose that there exists a homeomorphism
$h:X_A\longrightarrow X_B$ that gives rise to a 
continuous orbit equivalence between $(X_A,\sigma_A)$ and $(X_B, \sigma_B)$.
\begin{enumerate}
\renewcommand{\theenumi}{\roman{enumi}}
\renewcommand{\labelenumi}{\textup{(\theenumi)}}
\item
There exists an isomorphism $\varphi_h:G_A\longrightarrow G_B$
of \'etale groupoids such that 
$\varphi_h(G_{A,\Psi_h(g)})=G_{B,g}$ 
for all $g\in C(X_B,\Z)$,
where $\Psi_h(g)\in C(X_A,\Z)$ is defined by
\begin{equation}\label{eq:Psihg}
\Psi_h(g)(x) = \sum_{i=0}^{l_1(x)}g(\sigma_B^i(h(x))) -
\sum_{j=0}^{k_1(x)}g(\sigma_B^j(h(\sigma_A(x)))),
\qquad  x \in X_A.
\end{equation}  
In particular,  the \'etale groupoids
$G_{A,\Psi_h(g)}$ and $G_{B,g}$
are isomorphic.
\item
There exists an isomorphism $\Phi: \OA\longrightarrow \OB$ 
of $C^*$-algebras such that 
$\Phi(\DA) = \DB$ and $\Phi(\F_{A,\Psi_h(g)}) = \F_{B,g}$ for all $g \in C(X_B,\Z)$.
\end{enumerate}
\end{corollary}
\begin{proof}
(i) Assume that 
a homeomorphism
$h:X_A\longrightarrow X_B$ gives rise to a 
continuous orbit equivalence between $(X_A,\sigma_A)$ and $(X_B, \sigma_B)$. 
By \cite[Theorem 3.2]{MaMZ}, 
there exists an isomorphism $\Phi:\OA\longrightarrow \OB$ 
of $C^*$-algebras such that $\Phi(\DA) = \DB$ and 
$\Phi\circ \rho^{A,\Psi_h(g)}_t = \rho^{B,g}_t, t \in \T$.
Hence we see that   
$\Phi(\F_{A,\Psi_h(g)}) = \F_{B,g}$ for all $g \in C(X_B,\Z)$
so that by Proposition \ref{prop:coeandcocycle},
there exists an isomorphism $\varphi_h:G_A\longrightarrow G_B$
of \'etale groupoids such that 
$\varphi_h(G_{A,\Psi_h(g)})=G_{B,g}$ 
for all $g\in C(X_B,\Z).$

(ii) The assertion follows from (i) together with Proposition \ref{prop:coeandcocycle}.
\end{proof}

\begin{corollary}\label{cor:conjugacygroupoids}
One-sided topological Markov shifts
 $(X_A,\sigma_A)$ and $(X_B, \sigma_B)$
 are topologically conjugate if and only if there exists an isomorphism
 $\varphi:G_A\longrightarrow G_B$ of \'etale groupoids
 such that
 $\varphi(G_{A, g\circ h}) = G_{B,g}$ for all 
 $g \in C(X_B,\Z),$ 
where $h:X_A \longrightarrow X_B$
is a homeomorphism defined by the restriction of $\varphi$
to its unit space $G_A^{(0)}$ under the identification between 
$G_A^{(0)}$ and $X_A$, and  $G_B^{(0)}$ and $X_B$, respectively. 
\end{corollary}
\begin{proof}
By Theorem \ref{thm:1.2},
 $(X_A,\sigma_A)$ and $(X_B, \sigma_B)$
 are topologically conjugate if and only if there exists an isomorphism
$\Phi:\OA\longrightarrow \OB$ of $C^*$-algebras
such that $\Phi(\DA) = \DB$ and 
$\Phi(\F_{A,g\circ h}) = \F_{B,g}$
for all $g \in C(X_B,\Z).$
Hence the assertion follows from Proposition \ref{prop:coeandcocycle}.
\end{proof}


\section{Three classes of cocycle algebras}
We fix an irreducible non permutation matrix $A$ with entries in $\{0,1\}.$
For $b\in C(X_A,\Z)$ and $H\subset \{1,2,\dots,N\}$,
define continuous functions
$1_b, \chi_H \in C(X_A,\Z)$ by 
\begin{equation*}
1_b(x) = 1 - b(x) + b(\sigma_A(x)), \qquad
\chi_H(x) =
\begin{cases}
1 & \text{ if } x_1 \in H, \\ 
0 & \text{ if } x_1 \not\in H,  
\end{cases}
\end{equation*}
for  $x=(x_n)_{n\in \N}  \in X_A.$
For $x \in X_A$ and $k \in \N$,
we write
$x_{[k,\infty)} = \sigma_A^{k-1}(x) \in X_A.$
\begin{lemma}
For a nonempty subset $H \subset \{1,2,\dots,N\}$,
 a continuous function $b \in C(X_A,\Z)$
and a positive integer valued function $f \in C(X_A,\N)$,
we have 
\begin{enumerate}
\renewcommand{\theenumi}{\roman{enumi}}
\renewcommand{\labelenumi}{\textup{(\theenumi)}}
\item 
If $\chi_H = 1_b$, 
then 
$H= \{1,2,\dots, N\}$ and $b$ is a constant. 
\item 
If $1_b = f$, 
then 
  $b$ is a constant and $f \equiv 1$. 
\item 
If $f =\chi_H $, 
then 
$f \equiv 1$ and  $H= \{1,2,\dots, N\}$. 
\end{enumerate}
\end{lemma}
\begin{proof}
(i)
Suppose that $\chi_H = 1_b$.
For $x =(x_n)_{n\in \N} \in X_A$ and $k \in \N$,
we have
\begin{equation*}
b(x_{[k,\infty)}) =
\begin{cases}
b(x_{[k+1,\infty)}) & \text{ if } x_k \in H, \\
1 + b(x_{[k+1,\infty)}) & \text{ if } x_k \not\in H,
\end{cases}
\end{equation*}
so that  for $n \in \N$
\begin{equation*}
b(x) = N_{H^c}(x_1,\dots, x_n) + b(x_{[n+1,\infty)}),
\qquad x \in X_A,
\end{equation*}
holds where 
$N_{H^c}(x_1,\dots, x_n)  
= |\{i \mid x_i \not\in H, i=1,2,\dots, n \}|$.
Since $b:X_A \longrightarrow \Z$ is continuous and hence bounded,
so that $N_{H^c}(x_1,\dots, x_n)$ 
must be bounded  for all $x \in X_A$ and $n \in \N$.
As $A$ is irreducible, it does not occur unless
$H = \{1,2,\dots,N\}$.
Therefore we conclude that 
$H = \{1,2,\dots,N\}$, so that $b$ is a constant.

(ii)
Suppose that $1_b = f$ for some $f \in C(X_A, \N)$.
Since $f(x) \ge 1$ for any $x \in X_A$, we have
$1 -b(x) + b(\sigma_A(x)) \ge 1$ and hence
\begin{equation}\label{bsigma}
b(\sigma_A(x)) \ge b(x) \qquad \text{ for all } x \in X_A.
\end{equation} 
Assume that $b$ is not constant so that there exist
$y, z \in X_A$ such that 
\begin{equation}\label{eq:yz}
b(y) > b(z).
\end{equation}
Since $b \in C(X_A,\Z)$, one may find $K \in \N$ 
such that  
$$
b = \sum_{\nu \in B_K(X_A)} b_\nu \chi_{U_\nu} \qquad 
\text{ for some } b_\nu \in \Z \text{ for }\nu \in B_K(X_A).
$$
Hence there exist
$\mu_y, \mu_z \in B_K(X_A)$ such that 
$$
y \in U_{\mu_y},\quad z \in U_{\mu_z}\quad 
\text{ and }
\quad
b(y) = b_{\mu_y}, \quad b(z) = b_{\mu_z}.
$$
As the matrix $A$ is irreducible,
there exists $\xi \in B_M(X_A)$ such that 
$\mu_y \xi \mu_z \in B_{2K + M}(X_A)$,
so that 
$w := \mu_y \xi \mu_z z_{[k+1,\infty)}  \in X_A$.
As $w \in U_{\mu_y}$,
we have 
$b(w) = b_{\mu_y} = b(y)$.
On the other hand,
we have by \eqref{bsigma}
\begin{equation}\label{eq:bzby}
b(z) = b(\mu_z z_{[K+1,\infty)}) = b(\sigma_A^{K+M}(w)) \ge b(w) = b(y),
\end{equation} 
a contradiction to \eqref{eq:yz}.
We thus conclude that $b$ is a constant and hence $f \equiv 1$.

(iii)
Suppose that $\chi_H(x) = f(x) \ge 1 $ for all $x \in X_A$.
This implies that $H = \{1,2,\dots, N\}$.
\end{proof}

\subsection{Support subalgebras}

In this subsection, we will study the first class of cocycle algebras
called support algebras.
For $H \subset \{1,2,\dots, N\}$,
let us denote by $\rho^{A,H}$
the gauge action  $\rho^{A,\chi_H}$
with potential function $\chi_H$.
\begin{definition}
Define the $C^*$-subalgebra $\F_{A, H}$ of $\OA$
by  the cocycle algebra $\F_{A,\chi_H}$ for the funtion $\chi_H$,
that is defined by the fixed point subalgebra  of $\OA$ under the action $\rho^{A,H}$
\begin{equation}\label{eq:defFAH}
\F_{A,H} : = 
\{ X \in \OA \mid \rho^{A,H}_t(X) = X \text{ for all } t \in \T\}.
\end{equation}  
\end{definition}
The algebra $\F_{A,H}$ is called the {\it support algebra for}\/ $H$.
If $H = \{ 1, 2, \dots, N\},$ then $\rho^{A,H}_t =\rho^A_t$ 
so that $\F_{A,H} = \FA$.
If $H = \emptyset,$ then $\rho^{A,H}_t = \id $ so that $\F_{A,H} = \OA$.

For a word
$\mu = (\mu_1,\dots,\mu_m) \in B_m(X_A)$, 
let 
\begin{equation*}
N_H(\mu) := | \{ i \in \{ 1,\dots, m \} \mid \mu_i \in H \}|
\end{equation*}
the cardinal number of symbols in $\{\mu_1, \dots,\mu_m\} $ contained in $H$. 
\begin{lemma}\label{lem:4.3}
For $\mu,\nu \in B_*(X_A)$ such that $S_\mu S_\nu^* \ne 0$,
 we have 
$S_\mu S_\nu^* \in \F_{A,H}$ 
if and only if $N_H(\mu) = N_H(\nu)$.
\end{lemma}
\begin{proof}
Since $\chi_H$ is identified with $\sum_{j\in H}S_j S_j^*$,
for $\mu =(\mu_1,\dots,\mu_m) \in B_m(X_A)$
we have
\begin{equation*}
\chi_H S_\mu = 
\begin{cases}
S_\mu & \text{ if } \mu_1 \in H, \\
0 & \text{ otherwise,}
\end{cases}
\end{equation*} 
so that 
\begin{equation*}
S_{\mu_1} \chi_H S_{\mu_1}^* S_\mu 
= S_{\mu_1} \chi_H S_{\mu_2 \cdots \mu_m} 
= 
{\begin{cases}
S_\mu & \text{ if } \mu_2 \in H, \\
0 & \text{ otherwise.}
\end{cases}}
\end{equation*}
Similarly we have for $1 \le n \le m-1$
\begin{equation*}
S_{\mu_1\cdots\mu_n} \chi_H S_{\mu_1\cdots\mu_n}^* S_\mu 
= S_{\mu_1\cdots\mu_n} \chi_H S_{\mu_{n+1} \cdots \mu_m} 
= 
{\begin{cases}
S_\mu & \text{ if } \mu_{n+1} \in H, \\
0 & \text{ otherwise.}
\end{cases}}
\end{equation*}
By the equality
\begin{equation*}
\chi_H^{|\mu|} S_\mu = (
\chi_H + S_{\mu_1} \chi_H S_{\mu_1}^* + S_{\mu_1\mu_2} \chi_H S_{\mu_1\mu_2}^* + \cdots +
 S_{\mu_1\cdots\mu_{m-1}} \chi_H S_{\mu_1\cdots\mu_{m-1}}^* )S_\mu,
\end{equation*}
we have
\begin{equation*}
\chi_H^{|\mu|} S_\mu = N_H(\mu) S_\mu.
\end{equation*}
Similarly we have
$
S_\nu^* \chi_H^{|\nu|} = (\chi_H^{|\nu|} S_\nu)^*  = N_H(\nu) S_\nu^*.
$
We thus have
$
\chi_H^{|\mu|} S_\mu S_\nu^* =S_\mu S_\nu^* \chi_H^{|\nu|}$ 
if and only if 
$ N_H(\mu) S_\mu S_\nu^* =  N_H(\nu) S_\mu S_\nu^*.$
Assume that $S_\mu S_\nu^* \ne 0$.
 By Lemma \ref{lem:smusnufaf},  
$
S_\mu S_\nu^* \in \FAH
$
if and only if 
$N_H(\mu) = N_H(\nu).$
\end{proof}
By 
Proposition \ref{prop:generated}
and
Lemma \ref{lem:4.3},
we have the following proposition.
\begin{proposition}
For $H\subset \{1,2,\dots,N\}$,
 the support algebra $\F_{A,H}$ is the $C^*$-subalgebra 
$C^*(S_\mu S_\nu^* \mid N_H(\mu) = N_H(\nu), \mu, \nu \in B_*(X_A))$
of $\OA$ 
generated by partial isometries 
$
S_\mu S_\nu^*
$
satisfying $N_H(\mu) = N_H(\nu), \mu, \nu \in B_*(X_A)$.
\end{proposition}

We will next study a family of 
\'etale subgroupoids $G_H$ 
and its $C^*$-algebras $C^*(G_H)$.
\begin{definition}
For $H \subset \{1,2,\dots, N\}$,
define an  \'etale subgroupoid $G_H$ of $G_A$
by setting
\begin{align*}
G_H := \{ &
(x, n, z) \in X_A\times \Z \times X_A \mid  \text{there exist } k,l \in \Zp \text{ such that }\\
& n= k-l, \, \sigma_A^k(x) = \sigma_A^l(z), \,   
N_H(x_1,\dots, x_k) = N_H(z_1,\dots, z_l) \}.
\end{align*}
\end{definition}
Put the unit space 
$G_H^{(0)} = \{(x, 0, x) \in G_H\mid  x \in X_A \}.$
The product and the inverse operation are inherited from $G_A$.
\begin{lemma}
$G_{A,H} =G_{A, \chi_H}$ the cocycle groupoid for $\chi_H$. 
\end{lemma}
\begin{proof}
For $x = (x_n)_{n \in \N} \in X_A$, we have
\begin{equation*}
\chi_H^k(x) =\chi_H(x) + \chi_H(\sigma_A(x)) + \cdots + \chi_H(\sigma_A^{k-1}(x)) 
= N_H(x_1, x_2, \dots, x_k). 
\end{equation*}
Hence for $x =(x_n)_{n \in \N}, z=(z_n)_{n\in \N} \in X_A$
we have  
$\chi_H^k(x) = \chi_H^l(z)$ if and only if
$N_H(x_1,\dots, x_k) = N_H(z_1,\dots, z_l).$
Therefore we conclude that 
$G_{A,\chi_H} = G_{A,H}$.
\end{proof}
Hence by Proposition \ref{prop:essprinamenable} and Theorem \ref{thm:main1},
we know that 
$G_H$ is an essentially principal amenable  \'etale subgroupoid of $G_A$
such that $C^*(G_H) = \F_{A,H}$.

Let 
$\G_A=(V_A,E_A)$
be the directed graph associated with the matrix $A$.
It is defined in the following way.
 The vertex set $V_A$ is defined by $ \{1,2,\dots, N\}$,
 and the edge set $E_A$ is defined by the set of edge $(i,j) \in V_A\times V_A$  
satisfying $A(i,j) =1$ whose source vertex is the vertex $i$ 
and terminal vertex is the vertex $j$.
A path $(\mu_1,\dots,\mu_m) \in B_m(X_A)$
in the graph $\G_A$ is called a cycle 
if $\mu_1 =\mu_m$.
It is  equivalent to say that 
the word $(\mu_1,\dots,\mu_{m-1})$
is a periodic word in $B_*(X_A)$.

We henceforth assume that a subset $H\subset \{1,2,\dots, N\}$ is not empty.
\begin{definition}
A subset $H \subset \{1,2,\dots, N\}$ is said to be {\it saturated}\/ if 
any cycle in the graph $\G_A$ has a vertex in $V_H$.
This means that
if a word $(\mu_1,\dots,\mu_m) \in B_m(X_A)$
satisfies $\mu_1 = \mu_m$, then there exists $i \in \{1,2,\dots, m-1\}$
such that $\mu_i \in H$.
\end{definition}
We then have the following lemma.
\begin{lemma}\label{lem:satu}
Let $H\subset\{1,2,\dots,N\}$ be a nonempty subset.
The following assertions are equivalent. 
\begin{enumerate}
\renewcommand{\theenumi}{\roman{enumi}}
\renewcommand{\labelenumi}{\textup{(\theenumi)}}
\item 
$H$ is saturated.
\item The cardinality 
$|\{ \mu \in B_*(X_A) \mid N_H(\mu) = n \} | $ is finite
for all $n \in \N$.
\item The cardinality 
$|\{ \mu \in B_*(X_A) \mid N_H(\mu) = n \} | $ is finite
for some $n \in \N$.
\end{enumerate}
\end{lemma}
\begin{proof}
(i) $\Longrightarrow$ (ii):
Assume that $H$ is saturated in $\{1,2,\dots,N\}$.
Suppose that there exists $n \in \N$ such that 
the cardinality 
$|\{ \mu \in B_*(X_A) \mid N_H(\mu) = n \} | $ is infinite.
Hence for any $p \in \N$, there exists $q\in \N$ with $q >p$
such that there exists $\nu \in B_q(X_A)$ satisfying $N_H(\nu) = n$.
Since there exists at least one cycle in a word of length $N+1$,
there exists more than $n+1$ cycles in a word of length 
longer than  $(N+1)(n+1)$.
Now $H$ is saturated so that for $p > (N+1)(n+1)$
and $q >p$, any word $\nu \in B_q(X_A)$ satisfies 
$N_H(\nu) \ge n+1$,
a contradiction.

(ii) $\Longrightarrow$ (iii): The implication is clear.

(iii) $\Longrightarrow$ (i):
Assume the assertion (iii).
Suppose that $H$ is not saturated in $\{1,2,\dots, N\}$.
There exists a periodic word 
$w =(w_1, \dots, w_p) \in B_*(X_A)$ with $A(w_p, w_1) =1$
such that $w_i \not\in H$ for all $i=1,2,\dots, p$,
so that $N_H(w) =0$.
Since $A$ is irreducible and $H \ne \emptyset$,
one may find a word $\nu \in B_*(X_A)$ such that 
$N_H(\nu) =n$ and $\nu w \in B_*(X_A)$.
Put $\nu(m) = \nu \overbrace{w\cdots w}^{m\text{ times}} \in B_*(X_A)$
so that $N_H(\nu(m)) = N_H(\nu) = n$
because $N_H(w) =0$.
The family $\{\nu(m) \mid m=1,2,\dots \}$ is infinite, a contradiction. 
\end{proof}
Under the assumption that $H$ is saturated,
for each $n \in \Zp$
the set 
$\{ S_\mu S_\nu^* \mid N_H(\mu) = N_H(\nu) = n\}$
is finite because of Lemma \ref{lem:satu}.
Let us denote by
$\F_{A,H}^n$ the linear span of 
$\{ S_\mu S_\nu^* \mid N_H(\mu) = N_H(\nu) = n\}$.  
\begin{lemma}
For $S_\mu S_\nu^*, S_\xi S_\eta^* \in \F_{A,H}^n$,
we have
$S_\mu S_\nu^*\cdot S_\xi S_\eta^* \in F_{A,H}^n$.
Hence $\F_{A,H}^n$ is a finite dimensional $C^*$-subalgebra of $\OA$
if $H$ is saturated.
\end{lemma}
\begin{proof}
For $S_\mu S_\nu^*, S_\xi S_\eta^* \in \F_{A,H}^n$,
we may assume that 
$S_\mu S_\nu^*\cdot S_\xi S_\eta^*\ne 0$
and $|\nu|\le |\xi|$.
Since $S_\nu^* S_\xi \ne 0$,
we have $\xi = \nu\bar{\xi}$ for some $\bar{\xi} \in B_*(X_A)$.
As $N_H(\xi) = N_H(\nu) = n$,
we have $N_H(\bar{\xi}) =0$.
By putting $\nu = (\nu_1,\dots,\nu_p)$
and $\bar{\xi} = (\bar{\xi}_1,\dots,\bar{\xi}_q),$
we have
\begin{equation*}
S_\nu^* S_\xi 
=   S_\nu^* S_\nu S_{\bar{\xi}} 
=   S_{\nu_p}^* S_{\nu_p} S_{\bar{\xi}} 
=   \sum_{j=1}^N A(\nu_p, j)S_j S_j^* S_{\bar{\xi}} 
=   S_{\bar{\xi}}
\end{equation*}
so that 
\begin{equation*}
S_\mu S_\nu^*\cdot S_\xi S_\eta^*
= S_\mu S_{\bar{\xi}} S_\eta^*= S_{\mu \bar{\xi}} S_\eta^*.
\end{equation*}
As 
$N_H(\mu \bar{\xi}) = N_H(\mu) + N_H(\bar{\xi}) = N_H(\mu) = n = N_H(\eta)$,
we see that 
$S_\mu S_\nu^*\cdot S_\xi S_\eta^*$ belongs to $\F_{A,H}^n$.
If  $H$ is aturated,
the linear span of $\{ S_\mu S_\nu^* \mid N_H(\mu) = N_H(\nu) = n\}$
is  finite dimensional and 
closed under both multiplication and $*$-operation,
so that it is a finite dimensional $C^*$-algebra.
\end{proof}
\begin{lemma}\label{lem:icode}
Assume that $H$ is saturated.
For each $i \in \{1,2,\dots,N\}$,
there exists a finite family $\{ \omega^i(j)\}_{j=1}^{p_i}$
of words
 $\omega^i(j) = (\omega^i_1(j), \dots, \omega_{\ell(j)}^i(j)) 
 \in B_{\ell(j)}(X_A), j=1,2,\dots, p_i$ 
such that for each $j=1,2,\dots, p_i$,
\begin{gather*}
\omega^i_1(j) = i, \qquad 
\omega_{\ell(j)}^i(j) \in H, \qquad
 N_H(\omega^i(j)) = 1, \\
 U_{i} = \cup_{j=1}^{p_i} U_{\omega^i(j)},\qquad
  U_{\omega^i(j)}\cap  U_{\omega^i(j')} =\emptyset
  \text{ for } j\ne j'.
\end{gather*}
\end{lemma}
\begin{proof}
If $i \in H$, then $p_i =1$ and take 
$\omega^i(1)$ as $i$.
If $i \not\in H$,
take the set of words $(\omega^i_1(j), \dots, \omega_{\ell(j)}^i(j))$
starting with  $\omega^i_1(j)=i$ 
and ending with $\omega_{\ell(j)}^i(j)$ in $H$
such that 
$\omega^i_q(j) \not\in H$
for $q=1,2,\dots,  \ell(j)-1$.
Since $H$ is saturated, the family satisfies
$U_{i} = \cup_{j=1}^{p_i} U_{\omega^i(j)}$.
\end{proof} 
\begin{lemma}\label{lem:AFHinclusion}
Assume that $H$ is saturated.
\begin{enumerate}
\renewcommand{\theenumi}{\roman{enumi}}
\renewcommand{\labelenumi}{\textup{(\theenumi)}}
\item $\F_{A,H}^n \subset \F_{A,H}^{n+1}, \, n \in \N$.
\item 
$1 \in \F_{A,H}^1$ and hence $1 \in \F_{A,H}^n$ for all $n \in \N$,
where
$1$ is the unit of the $C^*$-algebra $\OA$. 
\end{enumerate}
\end{lemma}
\begin{proof}
Take the finite family $\{ \omega^i(j)\}_{j=1}^{p_i}$
of words as in Lemma \ref{lem:icode}, so that we have 
$$
 U_{i} = \cup_{j=1}^{p_i} U_{\omega^i(j)},\qquad
  U_{\omega^i(j)}\cap  U_{\omega^i(j')} =\emptyset
  \text{ for } j\ne j'.
$$
This means that
the equality
$S_iS_i^* = \sum_{j=1}^{p_i} S_{\omega^i(j)}S_{\omega^i(j)}^*
$ holds for $i=1,2,\dots,N$.
Hence we have
$$
1 =\sum_{i=1}^N S_iS_i^* = \sum_{i=1}^N \sum_{j=1}^{p_i} S_{\omega^i(j)}S_{\omega^i(j)}^*.
$$
Since $N_H(\omega^i(j) ) =1$, we have 
$S_{\omega^i(j)}S_{\omega^i(j)}^* \in \F_{A,H}^1$
so that $1 \in \F_{A,H}^1$.

For $\mu,\nu \in B_*(X_A)$ such that 
$N_H(\mu) = N_H(\nu) =n$, we have
$S_\mu S_\nu^* \in \F_{A,H}^n$ and
\begin{equation*}
S_\mu S_\nu^*
= \sum_{i=1}^N \sum_{j=1}^{p_i} S_\mu S_{\omega^i(j)}S_{\omega^i(j)}^* S_\nu^* 
=  \sum_{i=1}^N \sum_{j=1}^{p_i} S_{\mu \omega^i(j)} S_{\nu \omega^i(j)}^*.
\end{equation*}
As $N_H(\mu \omega^i(j)) = N_H(\mu) + N_H(\omega^i(j)) = n+1$
and similarly
$N_H(\nu \omega^i(j)) = n+1,$
we know that 
$S_\mu S_\nu^* \in \F_{A,H}^{n+1}$.
\end{proof}
Assume that $H$ is saturated in $\{1,2,\dots, N\}$.
Let us denote by $\Sigma_H^i$
the set $\{ \omega^i(j)\}_{j=1}^{p_i}$ in Lemma \ref{lem:icode}.
We set  
$$
\Sigma_H = \cup_{i=1}^N \Sigma_H^i =\cup_{i=1}^N \{ \omega^i(1), \dots, \omega^i(p_i)\}.
$$
Put $M = \sum_{i=1}^N p_i$ and denote by
$\{\omega(1), \dots,\omega(M)\}$
the set $\Sigma_H.$ 
Write each word $\omega(m)$ as
$$
\omega(m) = (\omega_1(m), \dots, \omega_{\ell(m)}(m)) \in B_{\ell(m)}(X_A), \qquad
m=1,2,\dots,M.
$$
We then have
for $m=1,2,\dots,M$
\begin{gather*}
\omega_{\ell(m)}(m) \in H, \qquad 
\omega_k(m) \not\in H \text{ for } k=1,2,\dots, \ell(m)-1,\\
X_A = \bigcup_{m=1}^M U_{\omega(m)}: \quad \text{ disjoint union}.
\end{gather*}
Define an $M \times M$ matrix $A_H$ with entries in $\{0,1\}$ 
by setting for $m,n \in \{1,2,\dots, M\}$
\begin{equation} \label{eq:matrixAH}
A_H(m,n)= A(\omega_{\ell(m)}(m), \omega_1(n))
=
\begin{cases}
1 & \text{ if } \omega(m)\omega(n) \in B_*(X_A),\\
0 & \text{ otherwise. }
\end{cases}
\end{equation}
We thus have the following proposition.
\begin{proposition}\label{prop:AHAF}
Assume that $H$ is saturated.
The $C^*$-algebra $\F_{A,H}$ is a unital AF-algebra defined 
by the inclusion matrix $A_H$. 
\end{proposition}
\begin{proof}
The subalgebras $\F_{A,H}^n, n \in \N$ are increasing sequence of 
finite dimensional $C^*$-algebras
with the common unit $1$ of $\OA$.
Since
$\F_{A,H} =\overline{\cup_{n=1}^\infty \F_{A,H}^n},$
one knows that 
$\F_{A,H} $ is a unital AF-algebra.
As in the proof of Lemma \ref{lem:AFHinclusion},
the inclusion $\F_{A,H}^n \subset \F_{A,H}^{n+1}$
is given by the matrix $A_H$.
\end{proof}

\begin{definition} \label{def:primitive}
 A saturated subset $H \subset \{1,2,\dots, N\}$
is said to be {\it primitive}\/ if
the $M \times M$ matrix $A_H$ is primitive, 
that is,  there exists
$K \in \N$ such that $A_H^K(m,n) \ge 1$ for all $m,n=1,\dots,M$.
\end{definition}
It is well-known that an AF-algebra defined by an inclusion matrix is simple if and only if 
the matrix is primitive.
We thus have the following theorem by Proposition \ref{prop:AHAF}.
\begin{theorem}\label{thm:support}
Assume that $H$ is saturated in $\{1,2,\dots, N\}$.
The $C^*$-algebra $\F_{A,H}$ is a unital AF-algebra defined 
by the inclusion matrix $A_H$.
Furthermore, if $H$ is primitive, 
the $C^*$-algebra $\F_{A,H}$ is simple.
\end{theorem}

\medskip

Let us present several examples of the support algebras $\F_{A,H}$. 

{\bf 1.}
$A =
\begin{bmatrix}
1 & 1 \\
1 & 0 
\end{bmatrix},  
$
$H = \{1\} \subset \{1,2\}.$
It is easy to see that $H$ is saturated.
Put $\alpha_1 = 1, \alpha_2 = 21$
so that 
$\Sigma_H = \{\alpha_1,\alpha_2\}$.
Then we have 
$A_H =
\begin{bmatrix}
1 & 1 \\
1 & 1 
\end{bmatrix}
$
and hence $H$ is primitive.
We may regard $X_{A_H} \subset \{\alpha_1,\alpha_2\}^\N$.
Let $S_1, S_2$ be the canonical generating partial isometries of 
$\OA$.
The we have 
\begin{equation*}
\F_{A,H} 
=  C^*( S_\xi S_\eta^* \mid \xi, \eta \in B_*(X_{A_H}); |\xi| = |\eta| ) 
=  \F_{A_H}. 
\end{equation*}
Hence the $C^*$-algebra $\F_{A,H}$ 
is isomorphic to the UHF algebra $ M_{2^\infty}$ of type $2^\infty$. 

\medskip

{\bf 2.}
$A =
\begin{bmatrix}
0& 1 & 1 \\
1& 0 & 1 \\
1& 1 & 0
\end{bmatrix},  
$
$H = \{1, 2\} \subset \{1,2,3\}.$
It is easy to se that $H$ is saturated.
Put $\alpha_1 = 1, \alpha_2 = 2, \alpha_3 = 31, \alpha_4 = 32$,
so that
$\Sigma_H = \{\alpha_1, \alpha_2, \alpha_3,\alpha_4\}$.
Then we have 
$A_H =
\begin{bmatrix}
0& 1& 1 & 1 \\
1& 0& 1 & 1 \\
0& 1& 1 & 1 \\
1& 0& 1 & 1 
\end{bmatrix}
$
and hence $H$ is primitive.
We may regard $X_{A_H} \subset \{\alpha_1, \alpha_2,\alpha_3, \alpha_4\}^\N$.
Let $S_1, S_2, S_3$ be the canonical generating partial isometries of $\OA$.
Then we have 
\begin{equation*}
\F_{A,H} 
=  C^*( S_\xi S_\eta^* \mid 
\xi, \eta \in B_*(X_{A_H}), |\xi| = |\eta| ) 
=  \F_{A_H} 
\end{equation*}
Hence the $C^*$-algebra $\F_{A,H}$ is isomorphicto the simple AF-algebra $\F_{A_H}$. 

\medskip



{\bf 3.}
%
$A =
\begin{bmatrix}
1 & 1 \\
1 & 1 
\end{bmatrix},  
$
that is denoted by $[2]$,
and
$H = \{ 1 \} \subset \{ 1,2\}$.
In the one-sided topological Markov shift
$(X_{[2]}, \sigma_{[2]})$,
$z=2^\infty = (2,2,2,\dots ) \in X_{[2]}$ is the fixed point, 
but the symbol $2$ does not belong to $H$,
so that $H$ is not saturated in $\{1,2\}.$ 
It is easy to see that 
$\{ z \}$ is a $G_H$-invariant subset, so that 
the point corresponds to a closed ideal of the $C^*$-algebra
$C^*(G_H)$.
It is a proper ideal of $C^*(G_H)$.    
Hence the $C^*$-algebra $\F_{[2], \{1\}}$ is not simple.

\subsection{Coboundary algebras}
In this subsection, we will study the second class of cocycle algebras
called coboundary algebras from a view point of \'etale groupoids.
\begin{lemma}
Suppose that $f = 1 - b + b\circ \sigma_A$ for some $b \in C(X_A,\Z)$.
For $(x, n, z) \in G_A$, we have
$(x, n,z) \in G_{A,f}$ if and only if
$n = b(x) - b(z)$.
\end{lemma}
\begin{proof}
Suppose that 
 $(x, n, z) \in G_{A,f}$.
Take $k,l \in \Zp$ such that 
$n = k-l, \, \sigma_A^k(x) = \sigma_A^l(z)$
and
$f^k(x) = f^l(z)$.
Since
$f^k(x) = k - b(x) + b(\sigma_A^k(x))$
and
$
 f^l(z) = l - b(z) + b(\sigma_A^l(z)),$
we have 
$k -l = b(x) - b(z)$.

Conversely,
suppose that 
$n = b(x) - b(z)$.
Since
 $(x, n, z) \in G_A,$
there exist $k,l \in \Zp$ such that 
$n = k-l, \, \sigma_A^k(x) = \sigma_A^l(z).$
As
$f^k(x) = k - b(x) + b(\sigma_A^k(x))$
and
$ f^l(z) = l - b(z) + b(\sigma_A^l(z)),$
the condition
$n = b(x) - b(z)$
implies that 
$f^k(x) = f^l(z)$
and hence
 $(x, n, z) \in G_{A,f}$.
\end{proof}
\begin{definition}
For $b \in C(X_A,\Z)$, define an  \'etale subgroupoid $G_A^b$ 
of $G_A$ by setting
\begin{align*}
G_A^b := \{ &
(x, n, z) \in X_A\times \Z \times X_A \mid  \text{there exist } k,l \in \Zp \text{ such that }\\
& n= k-l= b(x) - b(z), \, \sigma_A^k(x) = \sigma_A^l(z)  \}.
\end{align*}
\end{definition}
The unit space $(G_A^b)^{(0)}$ is defined by
$(G_A^b)^{(0)} =  \{(x,0,x) \in G_A^b \mid x \in X_A \}$.
The product and the inverse operation are inherited from $G_A$.
Since by putting
$1_b = 1 - b+ b\circ\sigma_A \in C(X_A,\Z)$,
 under the condition
$\sigma_A^k(x) = \sigma_A^l(z),$
we have
$1_b^k(x) = 1_b^l(z)$ if and only if 
$k-l = b(x) - b(z)$.
Hence we see 
\begin{lemma}\label{lem:3.18}
$G_A^b = G_{A, 1_b}$ the cocycle groupoid for $1_b$.
\end{lemma}
The cocycle groupoid $G_A^b$ is called the 
{\it coboundary groupoid for}\/ $b$.
By Proposition \ref{prop:essprinamenable}, Theorem \ref{thm:main1}
and Lemma \ref{lem:3.18}, 
we know that 
the coboundary groupoid $G_A^b$
is an essentially principal amenable  \'etale clopen subgroupoid of $G_A$.
\begin{lemma}
Assume that $A$ is primitive.
Then the function $1_b$ for each  $b \in C(X_A,\Z)$ is minimal.
Hence the coboundary groupoid $G_A^b$ is minimal.
\end{lemma}
\begin{proof}
Take arbitrary  $z \in X_A$ and $\mu=(\mu_1,\dots,\mu_m)\in B_m(X_A)$. 
Since $b \in C(X_A,\Z)$, one may find $K \in \N$ with $K\ge m$ 
such that  
$$
b = \sum_{\nu \in B_K(X_A)} b_\nu \chi_{U_\nu} \qquad 
\text{ for some } b_\nu \in \Z \text{ for }\nu \in B_K(X_A).
$$
Since $X_A = \cup_{\nu \in B_K(X_A)} U_\nu$,
one may find $\nu_z \in B_K(X_A)$ such that 
$z \in U_{\nu_z}$ so that
$b(z) = b_{\nu_z}$.
Since $|\mu| = m\le K$, 
there exists $\nu_0= (\nu_{01},\nu_{02}, \dots, \nu_{0K}) \in B_K(X_A)$
such that $U_{\nu_0} \subset U_\mu$.
We see that 
$b(x) = b_{\nu_0}$ for all $x \in U_{\nu_0}$.
 Now $A$ is primitive, so that 
there exists  $L \in \N$ such that $A^L(i,j)\ge 1$ 
for all $i,j=1,2,\dots,N$.

We have three cases.

Case 1, $b_{\nu_0} = b_{\nu_z}$: 
We have $b(x) = b(z).$ 
Since $A^L(\nu_{0K}, z_{K+L}) \ge 1,$
one finds a word  
$(\nu_{K+1}, \nu_{K+2},\dots,\nu_{K+L-1}) \in B_{L-1}(X_A)$
such that 
$A(\nu_{0K}, \nu_{K+1})= A(\nu_{K+L-1}, z_{K+L}) =1.$
Put
$$
x_{[1,K]} = \nu_0,\qquad 
x_{[K+1, K+L-1]} = (\nu_{K+1}, \dots,\nu_{K+L-1}), \qquad
x_{[K+L,\infty)} = z_{[K+L,\infty)}
$$ 
so that we have
$x =(x_n)_{n\in \N} \in U_{\nu_0} \subset X_A$
and
hence $b(x) = b_{\nu_0}$.
We then have
$$
\sigma_A^{K+L-1}(x) = \sigma_A^{K+L-1}(z) = z_{[K+L,\infty)}.
$$
By putting $k = l= K+L-1$,
we have
$x \in U_\mu$ and
$$
\sigma_A^k(x) = \sigma_A^l(z), \qquad k-l = b(x) - b(z) (=0). 
$$

Case 2, $b_{\nu_0} > b_{\nu_z}$: 
We put $M = b(x) - b(z) (= b_{\nu_0} - b_{\nu_z}) >0.$ 
Take 
$\nu' =(\nu'_1,\dots,\nu'_M) \in B_M(X_A)$
and 
$\nu^{\prime\prime} =(\nu^{\prime\prime}_{K+1},\dots,\nu^{\prime\prime}_{K+L-1}) \in B_{L-1}(X_A)$
 such that
$
A(\nu_{0K}, \nu'_1) 
=A(\nu'_M, \nu^{\prime\prime}_{K+1})
 = A(\nu^{\prime\prime}_{K+L-1}, z_{K+L})=1.
$
Put
\begin{gather*}
x_{[1,K]} = \nu_0,\qquad 
x_{[K+1, K+M]} = (\nu'_1, \dots,\nu'_M), \qquad
x_{[K+M+1, K+M+L-1]} = (\nu^{\prime\prime}_{K+1}, \dots,\nu^{\prime\prime}_{K+L-1}), \\
x_{[K+M+L,\infty)} = z_{[K+L,\infty)}
\end{gather*} 
so that we have
$x =(x_n)_{n\in \N} \in U_{\nu_0} \subset X_A$
and
$$
\sigma_A^{K+L+M-1}(x) 
= x_{[K+M+L,\infty)} 
= z_{[K+L,\infty)}
=\sigma_A^{K+L-1}(z).
$$
By putting $k = K+L + M -1$ 
and $l = K+L-1$,
we have
$x \in U_\mu$ and
$$
\sigma_A^k(x) = \sigma_A^l(z), \qquad k-l =M =  b(x) - b(z). 
$$

Case 3, $b_{\nu_0} < b_{\nu_z}$: 
We put $M' = b(z) - b(x) (= b_{\nu_z} -b_{\nu_0} ) >0.$ 
One may find $(\xi_1,\dots, \xi_{L-1}) \in B_{L-1}(X_A)$
such that 
$
A(\nu_{0K}, \xi_1) 
=A(\xi_{L-1}, z_{K+M'+L})=1.
$
Put
$$
x_{[1,K]} = \nu_0,\qquad 
x_{[K+1, K+L-1]} = (\xi_1, \dots,\xi_{L-1}), \qquad
x_{[K+L,\infty)} = z_{[K+M'+ L,\infty)}
$$ 
so that we have
$x =(x_n)_{n\in \N} \in U_{\nu_0} \subset X_A$
and
$$
\sigma_A^{K+L-1}(x) 
= x_{[K+L,\infty)} 
= z_{[K+M'+L,\infty)}
=\sigma_A^{K+M'+L-1}(z).
$$
By putting $k = K+L  -1$ 
and $l = K+M'+L-1$,
we have
$$
\sigma_A^k(x) = \sigma_A^l(z), \qquad k-l =-M' =  b(x) - b(z). 
$$
Therefore the function $1_b = 1 -b + b\circ\sigma_A$
is minimal. 
\end{proof}
\begin{definition}
For $b \in C(X_A,\Z)$, the {\it coboundary algebra}\/
  $\F_A^b$ is defined by the cocycle algebra 
  $\F_{A,1_b}$ for the function $1_b$.
\end{definition}
By Lemma \ref{lem:3.18},
$\F_A^b$ is isomorphic to the $C^*$-algebra $C^*(G_A^b)$ of the 
coboundary groupoid $G_A^b$.

\begin{proposition}
Assume that $A$ is primitive.
For any $b \in C(X_A,\Z)$, the coboundary algebra
$\F_A^b$ is simple.
\end{proposition}
\begin{proof}
By the previous lemma,
the coboundary groupoid $G_A^b$ is minimal,
so that the $C^*$algebra $C^*(G_A^b)$ that is $\F_A^b$
is simple.
\end{proof}
\begin{theorem}\label{thm:coboundary}
Assume that $A$ is primitive.
For any $b \in C(X_A,\Z)$, the coboundary algebra
$\F_A^b$ is a unital simple AF-algebra that is stably isomorphic to
the standard AF-algebra $\FA$.
\end{theorem}
\begin{proof}
The function $1_b = 1 - b + b\circ \sigma_A$
 is cohomologous to $1$, so that by putting a one-parameter unitary group
$u_t = \exp({2 \pi\sqrt{-1} b t}) \in \DA, t \in \T,$ 
we have by \cite[Lemma 2.3]{MaMZ}
\begin{equation*}
\rho^{A, 1_b}_t = \Ad(u_t) \circ \rho^A_t, \qquad t \in \T.
\end{equation*}
This means that the actions
$\rho^{A, 1_b}$ and $\rho^A$ are cocycle conjugate.
By general theory of $C^*$-crossed products, 
we know that their crossed products
$\OA\rtimes_{\rho^{A,1_b}}\T$ and
$\OA\rtimes_{\rho^{A}}\T$ 
are isomorphic.
By \cite{C2}, we know that 
$\OA\rtimes_{\rho^{A}}\T$ is stably isomorphic to $\FA$.
Now $A$ is primitive, so that the AF-algebra
$\FA$ is simple and hence 
so is $\OA\rtimes_{\rho^{A}}\T$.
Hence 
the crossed product
$\OA\rtimes_{\rho^{A,1_B}}\T$ is simple.  
By \cite{Rosenberg},  
there exists a projection $P_b \in  \OA\rtimes_{\rho^{A,1_b}}\T$
such that 
$(\OA)^{\rho^{A,1_b}} = P_b(\OA\rtimes_{\rho^{A,1_b}}\T)P_b.$
As 
the algebra $\F_A^b$ is the fixed point algebra
$(\OA)^{\rho^{A,1_b}}$ under the action
$\rho^{A,1_b}$,
it is a full corner of $\OA\rtimes_{\rho^{A,1_b}}\T$.
Hence 
$\F_A^b$ is stably isomorphic to $\OA\rtimes_{\rho^{A,1_b}}\T$
and hence  to the standard AF-algebra $\F_A$.
\end{proof}
In \cite{MaPacific}, the continuous full group
$\Gamma_A$, writtten as $[\sigma_A]_c$  in \cite{MaPacific}, 
for a topological Markov shift
$(X_A,\sigma_A)$ plays a crucial role to study continuous orbit equivalence 
in one-sided topological Markov shifts 
(see also \cite{MMGGD}, \cite{MatuiCrelle}, etc.).
The group $\Gamma_A$ consists of homeomorphisms $\tau$ on $X_A$ 
such that there exist $k_\tau, l_\tau \in C(X_A,\Zp)$ satisfying
$\sigma_A^{k_\tau(x)}(\tau(x)) = \sigma_A^{l_\tau(x)}(x)$ for $x \in X_A$.
The function $d_\tau = l_\tau -k_\tau \in C(X_A,\Z)$
is called the cocycle functin of $\tau$.
Since the homeomorphism $\tau$ 
gives rise to a continuous orbit equivalence on $X_A$,
one may consider the function $\Psi_\tau(1)\in C(X_A,\Z)$
 for $1$ by the formula \eqref{eq:Psihg}.
It is straightforward to see that 
the formula 
\begin{equation}\label{eq:Psitau1}
\Psi_\tau(1)(x) 
= 1 - d_\tau(x) + d_\tau\circ\sigma_A(x), \qquad x \in X_A
\end{equation}
holds (cf. \cite{MaPre2020c}).
We then know the following proposition. 
\begin{proposition}
Let $A$ be an irreducible, non permutation matrix with entries in $\{0,1\}$.
 The coboundary algebra $\F_A^{d_\tau}$ for the cocycle function $d_\tau= l_\tau -k_\tau$
of an element $\tau$ of the continuous full group 
$\Gamma_A$ is isomorphic to the standard AF-algebra
$\FA$.
\end{proposition}
\begin{proof}
Take an arbitrary $\tau \in \Gamma_A$.
By Corollary  \ref{cor:3.11} (ii),
there exists an isomorphism
$\Phi_\tau:\F_{A,\Psi_\tau(1)}\longrightarrow \F_{A,1}$ of $C^*$-algebras.
As $\Psi_\tau(1) = 1 - d_\tau + d_\tau\circ\sigma_A$ by \eqref{eq:Psitau1},
 we see that 
$\F_A^{d_\tau}$ is isomorphic to $\FA$.   
\end{proof}
There are lots of examples of homeomorphisms in the group $\Gamma_A$
as in \cite{MaDCDS}, \cite{MMGGD} and \cite{MatuiCrelle}.

\subsection{Suspension algebras}
In this subsection, we will study the third  class of cocycle algebras
called suspesion  algebras from a view point of suspension of two-sided subshifts.
 We fix an irreducible, non permutation matrix $A =A[i,j]_{i,j=1}^N$ with entries in $\{0,1\}$.
Let us denote by $(\bar{X}_A,\bar{\sigma}_A)$ 
the two-sided topological Markov shift
for  the matrix $A$
that is defined by the shift space 
$$
\bar{X}_A =\{ (x_n)_{n \in \Z} \in \{1,2,\dots, N\}^\Z \mid
A(x_n,x_{n+1}) =1 \text{ for all } n \in \Z\}
$$ 
and the shift homeomorphism 
$\bar{\sigma}_A((x_n)_{n \in \Z}) = (x_{n+1})_{n \in \Z} $
on $\bar{X}_A$.
Let $f:X_A\longrightarrow \N$ be a continuous function on $X_A$. 
By taking a higher block representation of $(X_A, \sigma_A$) (cf. \cite{LM}),
we may assume that the function $f$ depends only on the first coordinate 
on $X_A$. 
This means that there exists a finte family 
$f_1, \dots, f_N$ of positive integers
such that $f(x) = f_{x_1}$ for $x = (x_n)_{n\in \N}$.
Put 
$m_j = f_j -1 \in \Zp, j=1,\dots,N.$ 
Let $\bar{f}$ be the continuous function on $\bar{X}_A$ 
defined by 
$\bar{f}((x_n)_{n\in \Z}) =f( (x_n)_{n\in \N})$
for $(x_n)_{n\in \Z} \in \bar{X}_A$.

Let us denote by 
$(\bar{X}_A^f, S_A^f)$ the discrete suspension of $(\bar{X}_A, \bar{\sigma}_A)$
by the ceiling function $\bar{f}$, that is defined in the following way
(cf. \cite{BH}, \cite{PT}).
The space $\bar{X}_A^f$ is defined by
\begin{equation*}
\bar{X}_A^f = \{ (x, n) \in \bar{X}_A \times \Zp \mid  0\le n < \bar{f}(x) \},
\end{equation*}
and the transformation
$S_A^f$ on $\bar{X}_A^f$ is defined by for $(x,n) \in \bar{X}_A^f,$
\begin{equation*}
S_A^f(x,n) = 
\begin{cases}
(x, n+1) & \text{ if } n < \bar{f}(x) -1, \\
(\bar{\sigma}_A(x), 0) & \text{ if } n = \bar{f}(x) -1.
\end{cases}
\end{equation*}
Recall that  $\G_A = (V_A, E_A)$ 
denotes the directed graph defined by the matrix $A$
such that its vertex set $V_A =\{1,2,\dots, N\}$ and its edge set $E_A$
consists of the ordered pair $(i,j)$ of vertices such that $A(i,j) =1$. 
The source $s(i,j)$ of $(i,j)$ is $i$, and 
the terminal  $t(i,j)$ of $(i,j)$ is $j$. 
As in \cite{BH}, \cite{MaPre2018}, \cite{PT}, 
let us construct a new directed graph
$\G_{A_f} =(V_{A_f}, E_{A_f})$ with its transition matrix $A_f$
in the following way .
Let
$V_{A_f} = \cup_{j=1}^N \{j_0, j_1, \dots, j_{m_j}\}.$
For $j,k \in \{1,2,\dots,N\}$ with
$A(j,k) =1$, we define  
\begin{equation*}
A_f(j_0, j_1) = A_f(j_1, j_2) = \cdots = A_f(j_{m_j -1}, j_{m_j}) = 
A_f(j_{m_j}, k_0) = 1.
\end{equation*}
We define
$A_f(j_m, k_n) = 0$ for other pair 
$(j_m, k_n) \in V_{A_f} \times V_{A_f}$. 
We call the matrix $A_f$ the {\it suspended matrix}\/ of $A$ by $f$.

Denote by $\bar{X}_{A_f}$ the two-sided shift space defined by the suspended 
matrix $A_f$.  
Let $\Sigma_{A_f} =V_{A_f}$.
Hence we have
\begin{equation*}
\bar{X}_{A_f} =
\{ (x^f_i)_{i\in \Z} \in \Sigma_{A_f}^\Z \mid
x^f_i \in \Sigma_{A_f}, A_f(x^f_i, x^f_{i+1}) = 1\text{ for all } i \in \Z \}.
\end{equation*}
As $x^f_i = j(i)_{n(i)}$ for some $j(i) \in \{1,2,\dots, N\}$ and 
$n(i) \in \{0,1,\dots,m_{j(i)} \}$.
Let $x_1 = j(1)$.
Let $x_2$ be the first return time of $j(1)_{n(1)}$ to the $j_0$ in $(x^f_i)_{i\in \Z}$.
Similarly we have a sequence $(x_k)_{k\in \Z} \in \bar{X}_A$
of the first return time in  both forward and backward in $(x^f_i)_{i\in \Z}$.
We then have a correspondence
\begin{equation*}
\eta : \bar{X}_{A_f} \longrightarrow \bar{X}_A^f
\end{equation*}
such that 
$
\eta((x^f_i)_{i\in \Z}) = ((x_k)_{k\in \Z}, n(1)) \in \bar{X}_A^f.$
The following lemma is well-known.
\begin{lemma}[{cf. \cite{BH}, \cite{MaPre2018}, \cite{PS}, \cite{PT}}] 
For $f \in C(X_A,\N)$,
the map $\eta : \bar{X}_{A_f} \longrightarrow \bar{X}_A^f$
is a homeomorphism
satisfying $\eta\circ \bar{\sigma}_A = S_A^f \circ \eta$.
\end{lemma}
Hence the discrete suspension 
$(\bar{X}^f_A, S_A^f)$ is identified with the 
two-sided topological Markov shift
$(\bar{X}_{A_f}, \bar{\sigma}_{A_f})$
defined by the suspended matrix $A_f$. 

We will next study a crucial relation between these two $C^*$-algebras
$\F_{A_f}$ and $\FAf$.
Let us denote by 
$\tS_{j_0}, \tS_{j_1}, \dots, \tS_{j_{m_j}}, j=1,2,\dots,N$ 
the canonical generating partial isometries of the Cuntz--Krieger algebra
$\OAf$ so that 
\begin{gather*}
\sum_{j=1}^N (\tS_{j_0}\tS_{j_0}^* +\tS_{j_1}\tS_{j_1}^* +\cdots +\tS_{j_{m_j}}\tS_{j_{m_j}}^*) =1, \\
\tS_{j_n}^*\tS_{j_n} =
{\begin{cases}
 \tS_{j_{n+1}}\tS_{j_{n+1}}^* & \text{ for } n=0,1,\dots,m_j-1, \\
 \sum_{k=1}^N A(j,k) \tS_{k_0}\tS_{k_0}^* & \text{ for } n = m_j. 
\end{cases} 
}
\end{gather*}
Let us denote by 
$\rho^{A_f}_t, t\in \T$ the standard gauge action on
 $\OAf$ that is defined by
$$
\rho^{A_f}_t(\tS_{j_n}) =\exp{(2\pi\sqrt{-1}t)} \tS_{j_n}, 
\qquad n=0,1,\dots, m_j, \, j=1,2,\dots, N. 
$$
Let $\F_{A_f}$ be the standard AF-algebra inside $\OAf$
that is realized as the fixed point algebra of $\OAf$ 
under the gauge action $\rho^{A_f}$.
Define partial isometries $s_1,\dots, s_N$
and projection $P_A$ by
\begin{equation*}
s_j:= \tS_{j_0}\tS_{j_1} \cdots \tS_{j_{m_j}}, \qquad
P_A: = \sum_{j=1}^N \tS_{j_0}\tS_{j_0}^*
\end{equation*}
in $\OAf$. 
The following lemma is similarly shown to 
\cite[Section 6]{MaPre2018}.
\begin{lemma}\label{lem:fullOA}
Keep the above notation.
\begin{enumerate}
\renewcommand{\theenumi}{\roman{enumi}}
\renewcommand{\labelenumi}{\textup{(\theenumi)}}
\item
$\sum_{j=1}^N s_j s_j^* =P_A, \, s_i^* s_i = \sum_{k=1}^N A(j,k)s_ks_k^*, i=1,2,\dots,N$
\item $P_A\OAf P_A = C^*(s_1,\dots, s_N)$ the $C^*$-subalgebra of $\OAf$
generated by $s_1,\dots, s_N$. 
Hence  $P_A \OAf P_A$ is isomorphic to $\OA$.
\item $\rho^{A_f}_t(P_A) = P_A, \, t \in \T.$
\item Let $\Phi_A:P_A\OAf P_A\longrightarrow\OA=C^*(s_1,\dots, s_N)$
be the identification
$\Phi_A(\tS_{j_0}\tS_{j_1} \cdots \tS_{j_{m_j}}) = s_j$
in (ii)
 between $P_A\OAf P_A$ and $\OA$.
Then we have 
$\Phi_A\circ \rho^{A_f}_t|_{P_A\OAf P_A} = \rho^{A,f}_t\circ \Phi_A, \, t \in \T$.    
\end{enumerate}
\end{lemma}
Let us denote by $\K$ 
the $C^*$-algebra of compact operators on the separable 
infinite dimensional Hilbert space $\ell^2(\N)$,
and $\C$ its maximal abelian $C^*$-subalgebra of $\K$ consisting of
diagonal operators on $\ell^2(\N)$.  
By the above lemma, we see the following proposition.
\begin{proposition}\label{prop:fullFA}
\begin{enumerate}
\renewcommand{\theenumi}{\roman{enumi}}
\renewcommand{\labelenumi}{\textup{(\theenumi)}}
\item $P_A\F_{A_f} P_A = \FAf$.
\item Assume that $A$ is primitive.
Then $P_A$ is a full projection in $\F_{A_f}$,
 so that
the cocycle algebra  
$\FAf$ is stably isomorphic to $\F_{A_f}$.
More exactly, there exists an isomorphism
\begin{equation*} 
\Phi_f: \F_{A_f}\otimes\K \longrightarrow \FAf\otimes\K
\end{equation*}
of $C^*$-algebras such that 
$\Phi_f(\mathcal{D}_{A_f}\otimes\C) = \mathcal{D}_A\otimes\C.
$  
\end{enumerate}
\end{proposition}
\begin{proof}
(i)
Let $S_1, \dots, S_N$ be the canonical generating
partial isometries of $\OA$. 
Since the correspondence 
$\Phi_A: P_A \OAf P_A\longrightarrow \OA$
defined by $\Phi_A(\tS_{j_0}\tS_{j_1}\cdots \tS_{j_{m_j}}) = S_j$
is an isomorphism satisfying 
$\Phi_A\circ \rho^{A_f}_t = \rho^{A,f}_t\circ \Phi_A$,
we see that 
$(P_A \OAf P_A)^{\rho^{A_f}} = (\OA)^{\rho^{A,f}}$.
As 
$\rho^{A_f}_t(P_A) = P_A$,
we know that 
$P_A (\OAf)^{\rho^{A_f}} P_A= (\OA)^{\rho^{A,f}}$
and hence 
$P_A\F_{A_f} P_A = \FAf$.

(ii)
Now the matrix $A$ is primitive, 
so is the suspended matrix $A_f$.
Hence the AF-algebra $\F_{A_f}$ is simple.
As any nonzero projection of a simple $C^*$-algebra is full,
the projection $P_A$ is a full projection in $\F_{A_f}$.
It is easy to see that 
$P_A {\mathcal{D}}_{A_f} P_A = \DA$, so that 
the pair $(\F_{A,f}, \DA)$ is a relative full corner 
of the pair $(\F_{A_f}, \mathcal{D}_{A_f})$
in the sense of \cite{MaTAMS2018}.
Hence the pairs 
$(\F_{A_f}\otimes\K, \mathcal{D}_{A_f}\otimes\C)$
and
$(\F_{A,f}\otimes\K, \DA\otimes\C)$
are relative Morita equivalence in the sense of  \cite{MaTAMS2018},
so that   
there exists an isomorphism
\begin{equation*} 
\Phi_f: \F_{A_f}\otimes\K \longrightarrow \FAf\otimes\K
\end{equation*}
of $C^*$-algebras such that 
$\Phi_f(\mathcal{D}_{A_f}\otimes\C) = \mathcal{D}_A\otimes\C$  
by \cite{MaTAMS2018}.
\end{proof}
For a continuous function $f \in C(X_A,\N)$,
there exists $K \in \N$ such that 
$f = \sum_{\nu \in B_K(X_A)} f_\nu \chi_{U_\nu}$
for some $f_\nu \in \N$ for each $\nu \in B_K(X_A)$.
By taking the $K$-higher block matrix $A^{[K]}$ of $A$ as in \cite{LM},
the function $f$ is regarded as the one depending on only the first coordinate
of $X_{A^{[K]}}$.
We thus obtain the following theorem.
\begin{theorem}\label{thm:suspension}
Let $A$ be a primitive matrix with entries in $\{0,1\}$.
Let $f: X_A\longrightarrow \N$ be a continuous function on $X_A$.
Then the suspension algebra $\FAf$ is a unital simple AF-algebra
stably isomorphic to the AF-algebra defined by the suspended matrix $A^{[K]}_f$ 
of the $K$-higher block matrix of $A$ by $f$
for some $K\in \N$. 
\end{theorem}

\medskip

{\it Acknowledgment:}
This work was supported by JSPS KAKENHI 
Grant Number 19K03537.

\end{document}